\documentclass[a4paper,11pt]{amsart}

\usepackage[utf8]{inputenc}
\usepackage[T1]{fontenc}
\usepackage{lmodern}
\usepackage[english]{babel}
\usepackage{amssymb}
\usepackage{mathrsfs}
\usepackage{esint}
\usepackage{hyperref}
\usepackage[abbrev,backrefs]{amsrefs}
\usepackage[capitalize]{cleveref}
\usepackage{enumerate}
\usepackage{constants}

\usepackage{todonotes}

\newtheorem{proposition}{Proposition}[section]
\newtheorem{definition}[proposition]{Definition}
\newtheorem{lemma}[proposition]{Lemma}
\newtheorem{theorem}[proposition]{Theorem}
\newtheorem{corollary}[proposition]{Corollary}
\theoremstyle{definition}

\newtheorem{open}{Open problem}
\theoremstyle{remark}
\newtheorem{remark}[proposition]{Remark}

\newcommand{\defeq}{\triangleq}
\newcommand{\reals}{\mathbb{R}}
\newcommand{\targetspace}{V}
\newcommand{\dualspace}{{\targetspace}^\star}
\newcommand{\action}[2]{\langle{#1},{#2}\rangle}
\newcommand{\dualproduct}[2]{\action{#1}{#2}}
\newcommand{\Bigdualproduct}[2]{\Bigl\langle{#1},{#2}\Bigr\rangle}
\newcommand{\sphere}{\mathcal{S}}
\newcommand{\integers}{\mathbb{N}}
\DeclareMathOperator{\support}{supp}
\newcommand{\dif}{{\,\mathrm{d}}}
\newcommand{\genericnorm}[1]{\lvert{#1}\rvert}
\newcommand{\abs}[1]{\lvert {#1}\rvert}
\newcommand{\bigabs}[1]{\bigl\lvert {#1}\bigr\rvert}

\newcommand{\norm}[2]{\genericnorm{#2}_{#1}}

\newcommand{\biggnorm}[2]{\biggl\lvert #2 \biggr\rvert_{#1}}

\DeclareMathOperator{\diam}{diam}
\newcommand{\closure}[1]{\overline{#1}}
\newcommand{\chull}[1]{\operatorname{conv}{#1}}
\newcommand{\cchull}[1]{\closure{\chull{#1}}}
\newcommand{\borel}[1]{\mathscr{B}({#1})}

\newcommand{\hausdist}{d_{\mathfrak{H}}}

\newcommand{\st}{\,:\,}

\usepackage{ifthen}
\newcommand{\variation}[2]{%
\ifthenelse{\equal{#1}{\reals}}%
{\vert{#2}\vert}%
{\vert{#2}\vert_{#1}}%
}
\newcommand{\range}[2][]{%
\ifthenelse{\equal{#1}{}}{\operatorname{rg}#2 \left(X\right)}%
{\operatorname{rg} #2\left(#1\right)}%
}
\newcommand{\scalarproduct}[3]{%
\ifthenelse{\equal{#3}{\star}}%
{{#1}\star{#2}}%
{\left(#1 , #2\right)_{#3}}%
}
\numberwithin{equation}{section}

\newcommand{\mass}[2]{\variation{#1}{#2}(X)}

\title[Convergence and mass of vector measures]{Range convergence monotonicity for vector measures and range monotonicity of the mass}

\keywords{Wide convergence of vector measures; total variation; strict convexity; zonoids.}

\subjclass[2010]{46G10 (28B05, 52A05,52A39)}
\author{Justin Dekeyser}
\author{Jean Van Schaftingen}

\email{Jean.VanSchaftingen@UCLouvain.be}

\address{Universit\'e catholique de Louvain\\ 
Institut de Recherche en Math\'ematique et Physique\\
Chemin du Cyclotron 2 bte L7.01.01\\
1348 Louvain-la-Neuve\\
Belgium}

\thanks{J. Van Schaftingen was supported by the Mandat d'Impulsion Scientifique F.4523.17, ``Topological singularities of Sobolev maps'' of the Fonds de la Recherche Scientifique--FNRS}

\begin{document}

\begin{abstract}
We prove that the range of sequence of vector measures converging widely satisfies a weak lower semicontinuity property,
that the convergence of the range implies the strict convergence (convergence of the total variation) and that 
the strict convergence implies the range convergence for strictly convex norms.
In dimension $2$ and for Euclidean spaces of any dimensions, we prove that the total variation of a vector measure is monotone with respect to the range.
\end{abstract}

\maketitle
\tableofcontents

\section{Introduction}

A vector measure \(\mu\) maps some \(\Sigma\)--measurable subsets of a space \(X\) to vectors in a linear space \(V\) \citelist{\cite{Diestel_Uhl_1977}\cite{Halmos_1948}\cite{Ambrosio_Fusco_Pallara_2000}*{Chapter 1}}.
Vector measures appear naturally in calculus of variations and in geometric measure theory, as derivatives of functions of bounded variation (BV) \cite{DeGiorgi_1954} and as currents of finite mass \cite{Federer_1969}. 
In these contexts, it is natural to work with sequences of finite Radon measures and to study their convergence.

The \emph{wide convergence} of a sequence \((\mu_n)_{n \in \integers}\) of finite vector measures to some finite vector measure \(\mu\) is defined by the condition that for every compactly supported continuous test function \(\varphi \in C_c (X, \reals)\), one has 
\begin{equation*}
    \lim_{n \to \infty} \int_{X} \varphi \dif \mu_n
  = 
    \int_{X} \varphi \dif \mu
    .
\end{equation*}
The \emph{strict convergence}, which requires the total variation to converge also, is an important concept that brings one back into a setting similar to tight sequences of probability measures.

The \emph{total variation} of a vector measure \(\mu\) is defined in terms of a seminorm \(\abs{\cdot}_\targetspace\) on \(V\) by
\begin{equation*} 
  \variation{\targetspace}{\mu} (X)
  \defeq
    \sup
    \Bigg\{ 
      \sum_{E\in\mathcal{P}} \abs{\mu(E)}_{\targetspace} \st \mathcal{P}\ \text{is a finite \(\Sigma\)--measurable partition of}\ X \Bigg\} .
\end{equation*}
For real-valued signed Radon measures, \(\targetspace = \reals\) and the mass depends only, up to a constant positive multiple, on the vector structure of the space \(\targetspace\). 
This is not any more the case when \(V\) is a finite-dimensional vector space
of dimension greater than \(1\) (\(\dim \targetspace > 1\)). 
The norms on the space \(V\) are all equivalent, but induce total variations whose values are not connected mutually beyond the bounds given by equivalence of seminorms. 

This lack of connection between various total variations is the starting question of the present work. Reshetnyak~\cite{Reshetnyak} and Delladio~\cite{Delladio} have proved that strict convergence with respect to a strictly convex norm is a sufficient assumption to obtain strict convergence with respect to any another norm on \(\targetspace\) (see also \cite{Spector_2011} for the mass with respect to a Euclidean norm). 

In order to characterize strict convergence through a topological object, independent of a choice of norm, we propose to describe strict convergence with the \emph{range} of a vector measure \(\mu\), which is defined as 
\footnote{Classically, the range of a vector measure is defined~\cite{Diestel_Uhl_1977} as the closure of the set
\( 
  \{ \mu(A) : A \subset X \text{ is \(\Sigma\)--measurable} \} 
  \).
This definition coincides with ours as soon as \(\mu\) has no atoms  \citelist{\cite{Liapounoff_1940}\cite{Diestel_Uhl_1977}*{Corollary 9.5}}. Our notion has the advantage of being always a convex set and is natural in view of the link between range and the total variation \cite{Rodriguez_Piazza}*{Theorem 3}.
}
\begin{equation*} 
  \range{\mu} \defeq \cchull{\big\{\mu(A) : A \in\Sigma \big\}} .
\end{equation*}
In particular, when \(\targetspace = \reals\), we have 
\begin{equation}
\label{eq_range_1d}
 \range{\mu} 
 = [\mu^- (X), \mu^+ (X)],
\end{equation}
where \(\mu = \mu^+ + \mu^-\), with \(\mu^+ \ge 0\), \(\mu^- \le 0\), and \(\mu^+\) and \(\mu^-\) are mutually singular,
so that for any widely converging sequence of vector measures, the convergence of the range is equivalent with the convergence of the mass \cite{Conway}*{Appendix~C}.

In general, the range of a vector measure characterizes uniquely the total variation of a measure \cite{Rodriguez_Piazza}.
This result is stable under convergence as the range convergence implies in general the mass convergence:

\begin{theorem}
\label{theorem_range_implies_mass}
Let \(X\) be a locally compact separable metric space endowed with its Borel \(\sigma\)-algebra \(\mathscr{B}(X)\), let \(\targetspace\) be a finite-dimensional space and let \(\norm{\targetspace}{\cdot}\) be a seminorm on \(V\).
If \((\mu_n)_{n\in\integers}\) is a sequence of Borel vector measures 
from \(X\) to \(V\) that converges widely to some vector Borel measure \(\mu\),
and if the sequence of ranges \((\range{\mu_n})_{n\in\integers}\) converges to \(\range{\mu}\) in Hausdorff distance,
then
  \begin{equation*} 
    \lim_{n\to\infty}\mass{\targetspace}{\mu_n} 
    = \mass{\targetspace}{\mu} .
  \end{equation*}
\end{theorem}

The assumption on the convergence is described in terms of Hausdorff distance; this choice is in practice inconsequential: since the space \(V\) is finite-dimensional, the Hausdorff convergence of compact convex sets is equivalent to the Kuratowski convergence and to the convergence of the support functions \cite{Salinetti}*{Corollaries 3A and P4.A} (see also \cite{Lohne_Zalinescu_2006}).

The converse statement of \cref{theorem_range_implies_mass} does not hold in general.
In fact, the strict convexity of the norm appears as a necessary condition to have
range convergence:

\begin{proposition}
\label{proposition_strict_convexity}
If the locally compact separable metric space \(X\) is not discrete and if the seminorm \(\norm{\targetspace}{\cdot}\) is not strictly convex, then there exists a sequence \((\mu_n)_{n \in \integers}\) of Borel vector measures 
from \(X\) to \(V\) converging widely to some vector measure \(\mu\) such that \((\mass{\targetspace}{\mu_n})_{n \in \integers}\) converges to \(\mass{\targetspace}{\mu}\) but 
\((\range{\mu_n})_{n \in \integers}\) is constant and different from \(\range{\mu}\).
\end{proposition}

We prove that strict convexity is in fact a necessary and sufficient condition for
the equivalence between range convergence and mass convergence:

\begin{theorem}
\label{theorem_Euclidean}
Let \(X\) be a locally compact separable metric  space and \((\targetspace, \norm{\targetspace}{\cdot})\) be a finite-dimensional Euclidean space. If the
norm \(\norm{\targetspace}{\cdot}\) is strictly convex and if \((\mu_n)_{n\in\integers}\) is a sequence of Borel vector measures from \(X\) to \(V\) that converges widely to some vector measure \(\mu\),
then
  \begin{equation*} \lim_{n\to\infty}\range{\mu_n}=\range{\mu} \end{equation*}
if and only if
  \begin{equation*} \lim_{n\to\infty}\mass{\targetspace}{\mu_n} = \mass{\targetspace}{\mu} .\end{equation*}
\end{theorem}

The above result completely characterizes the strict convergence in terms of
range convergence. In comparison to mass convergence, the definition of the rank
depends on \(\targetspace\) only as a \emph{topological vector space},
and does not require any specific choice of norm.

The proof of \cref{theorem_Euclidean} follows a more general
approach due to Reshetnyak~\cite{Reshetnyak}, relying on
layerwise decomposition technique.

\medskip

\Cref{theorem_Euclidean} shows that mass convergence implies range convergence
in a strictly convex setting. It is also known~\cite{Rodriguez_Piazza} that
equality of the range implies equality of the variation. 
This leads to the question whether the mass of a measure depends monotonically on its range,
that is, whether the inclusion of sets \(\range{\nu} \subseteq \range{\mu}\)
does imply the inequality \(\mass{\targetspace}{\nu} \le \mass{\targetspace}{\mu}\).

In the case where the norm \(\norm{\targetspace}{\cdot}\) is Euclidean,
we prove such a relation:

\begin{theorem}
\label{mass_monotony}
Assume that \(\targetspace\) is a finite-dimensional vector space, equipped with a semi-norm \(\norm{\targetspace}{\cdot}\).
If either \(\dim V \le 2\) or \(\norm{\targetspace}{\cdot}\) is Euclidean, then 
if \(\Sigma\) is a \(\sigma\)--algebra on \(X\) and the vector measures \(\nu, \mu: \Sigma\to\targetspace\) satisfy
\begin{equation*} 
  \range{\nu} \subseteq \range{\mu}  ,
\end{equation*}
then
\begin{equation*} 
  \mass{\targetspace}{\nu} \le \mass{\targetspace}{\mu} .
\end{equation*}
If moreover the norm \(\norm{\targetspace}{\cdot}\) is strictly convex, equality in the above conclusion occurs if and only if the ranges
are equal.
\end{theorem}
In higher dimensions, we do not know if such relation still holds.
\begin{open}
If \(\dim \targetspace \ge 3\) and if \(\targetspace\) is endowed with an arbitrary
(respectively strictly) convex norm, is the mass (respectively strictly) monotone with respect
to range inclusion?
\end{open}
Our approach to prove \cref{mass_monotony} when \(V\) is a Euclidean space directly derives from the a representation
formula
\begin{equation}
\label{eq_ru9Ooph9ee}
      \norm{\targetspace}{v} 
    = 
      \int_{\sphere} \abs{\dualproduct{\eta}{v}}
      \dif\sigma(\eta)
  \end{equation}
for some suitable multiple of the uniform measure \(\sigma\) on the unit sphere \(\sphere\)
of \(\dualspace\). 
We will prove that \eqref{eq_ru9Ooph9ee}
implies the representation formula
\begin{theorem}
  \label{theorem_euclidean_representation}
  Assume that \(\targetspace\) is a finite-dimensional vector space, equipped with a Euclidean norm \(\norm{\targetspace}{\cdot}\), then for every \(\sigma\)--algebra \(\Sigma\) on \(X\), every vector measure \(\mu : \Sigma \to \targetspace\) and every \(A \in \Sigma\),
  \begin{equation*} 
  \variation{\targetspace}{\mu} (A)
  = 
  \int_{\sphere} \variation{\reals}{\dualproduct{\eta}{v}} (A)
  \dif\sigma(\eta).
  \end{equation*}
\end{theorem}

The representation formula \eqref{eq_ru9Ooph9ee} means that the norm is an \(L^1\)-norm in the sense of Choquet~\cite{Choquet}*{Paragraph~41},
and that the unit ball of the dual norm is a zonoid~\citelist{\cite{Choquet}\cite{Bolker_1969}*{Theorem 2.1}}.
In the two-dimensional case, any centrally symmetric convex set is a zonoid \citelist{\cite{Choquet}\cite{Bolker_1969}*{Theorem 5.4}} and, in fact, any norm has a similar representation for some positive even measure \(\sigma\) this implies a counterpart of \cref{theorem_euclidean_representation} for any two-dimensional norm. The latter characterization of zonoids fails in higher dimensions: for example the \(\ell^\infty\) norm whose 
whose dual \(\ell^1\) ball is hyperoctahedron, which is not a zonoid.

In dimension~2, one has similar representation
for any choice of norm, and strict convexity translates by density of the support
for the representing measure \(\sigma\).

\medskip

When the vector space \(V\) is two-dimensional, we improve
the geometric interplay between range convergence and mass convergence,
by drawing the following link:

\begin{theorem}
\label{theorem_perimeter_computation}
Assume that \(\targetspace\) is a two-dimensional vector space, equipped with a
seminorm \(\norm{\targetspace}{\cdot}\).
If \(\Sigma\) is a \(\sigma\)--algebra on \(X\) and if  \(\mu:\Sigma\to\targetspace\) is a vector measure, then
  \begin{equation*} \mass{\targetspace}{\mu} = \tfrac{1}{2}\,\operatorname{Per}_{\targetspace}\bigl(\range{\mu}\bigr) .
  \end{equation*}
\end{theorem}

Here \(\operatorname{Per}_{\targetspace} (\range{\mu})\) is the perimeter of the closed convex set \(\range{\mu}\),
computed with respect to the seminorm \(\norm{\targetspace}{\cdot}\).
When the target space \(\targetspace\) is equipped with a strictly convex norm,
the perimeter becomes strictly monotone and
\cref{theorem_Euclidean} may be derived from
\cref{theorem_perimeter_computation}.
This kind of isoperimetric inequalities for convex sets were also investigated in~\cite{Stefani_2018}.
\Cref{theorem_perimeter_computation} is the extension of the one-dimensional formula \(\mass{\reals}{\mu} = \diam \range{\mu}\) that follows from \eqref{eq_range_1d}.

\medskip

For the two-dimensional case,
\cref{theorem_Euclidean} may be derived also from the integral representation \eqref{eq_ru9Ooph9ee}.

\section{Total variation and range}

\subsection{Vector measures}

We assume that \(X\) is a space endowed with a \(\sigma\)--algebra \(\Sigma\) and that the space \(V\) is a finite-dimensional vector space on the real numbers. 
The space \(\targetspace^*\) is the algebraic dual of \(V\) and \(\dualproduct{\eta}{v} = \eta (v)\) denotes the duality product defined by the action of the linear form \(\eta \in \targetspace^*\) on the vector \(v \in \targetspace\).

A \emph{vector measure} \(\mu:\Sigma\to\targetspace\) is a countably additive function; when \(\targetspace = \reals\) it is a (finite) \emph{signed measure}.
A vector measure \(\mu\) induces definitions of integrals 
\begin{align*}
  \int_{A} f \dif \mu &\in \targetspace &
  & \text{ and } &
  \int_{A} \dualproduct{g}{\dif \mu} \in \reals.
\end{align*}
for all \(\Sigma\)--measurable and bounded functions \(f : X \to \reals\) and \(g : X \to \targetspace^*\) and sets \(A \in \Sigma\).
These integrals satisfy for every \(E \in \Sigma\) and \(\eta \in \targetspace^*\),
\begin{align*}
  \int_{A} \chi_E \dif \mu & = \mu (A \cap E) &
  & \text{ and } &
  \int_{A} \dualproduct{\eta\, \chi_E}{\dif \mu} & = \dualproduct{\eta}{\mu (A \cap E)},
\end{align*}
and the usual additivity, linearity and dominated convergence properties.

\subsection{Total variation}
The total variation of a vector measure is defined as follows (see \cite{Ambrosio_Fusco_Pallara_2000}*{Definition 1.4 (b)}). 

\begin{definition}
\label{definition_mass}
Let \(\Sigma\) be a \(\sigma\)--algebra on \(X\).
The total variation of a vector measure \(\mu : \Sigma \to X\) with respect to a seminorm \(\norm{\targetspace}{\cdot}\) is 
the function \(\variation{\targetspace}{\mu}: \Sigma \to [0, +\infty]\) defined for each set \(A \in \Sigma\) by 
  \begin{multline*}
 \variation{\targetspace}{\mu}(A)\\[-.5em]
  \shoveleft{\defeq 
  \sup\Bigg\{\sum_{n=1}^m\norm{\targetspace}{\mu(E_n)} :
  	E_1,\dots,E_m \in\Sigma, E_1,\dots,E_m\subseteq A}\\[-1.7em]
  \text{and the sets } E_1,\dots,E_m \text{are pairwise disjoint }\Bigg\} .
  \end{multline*}
\end{definition}

The total variation enjoys properties that make of it a finite measure \cite{Ambrosio_Fusco_Pallara_2000}*{Theorem 1.6}.

\begin{proposition}[Finiteness]
\label{proposition_variation_finite}
Let \(\Sigma\) be a \(\sigma\)--algebra on \(X\)
and let \(\mu : \Sigma \to V\) be a vector measure. For every \(A \in \Sigma\), \(\variation{\targetspace}{\mu} (A) < +\infty\). 
\end{proposition}

\begin{proposition}[Monotonicity]
\label{proposition_variation_monotone}
Let \(\Sigma\) be a \(\sigma\)--algebra on \(X\), let \(\mu:\Sigma\to\targetspace\) be a vector measure and let \(A, B\in\Sigma\).
If \(A \subseteq B\), then
  \begin{equation*} 
      \variation{\targetspace}{\mu}(A) 
    \subseteq 
      \variation{\targetspace}{\mu} (B)
      .
  \end{equation*}
\end{proposition}

\begin{proposition}[Countable additivity]
\label{proposition_total_variation_countable_additive}
Let \(\Sigma\) be a \(\sigma\)--algebra on \(X\) and let \(\mu:\Sigma\to\targetspace\) be a vector measure.
For every sequence \((A_n)_{n \in \integers}\) of disjoint sets in \(\Sigma\), 
\begin{equation*} 
  \variation{\targetspace}{\mu}{\bigl(\textstyle \bigcup_{n\in\integers}A_n\bigr)}
  =
  \sum_{n\in\integers} \variation{\targetspace}{\mu}(A_n).
\end{equation*}
\end{proposition}

We also have some algebraic properties:
\begin{proposition}[Triangle inequality]
\label{proposition_total_variation_triangle}
Let \(\Sigma\) be a \(\sigma\)--algebra on \(X\) and let \(\mu, \nu:\Sigma\to\targetspace\) be vector measures.
For every \(A\in\Sigma\),
  \begin{equation*}   
    \variation{\targetspace}{\mu + \nu}(A)
    \leq \variation{\targetspace}{\mu}(A) + \variation{\targetspace}{\mu}(A) .
  \end{equation*}
\end{proposition}

\begin{proposition}[Homogeneity]
\label{proposition_total_variation_homogeneous}
Let \(\Sigma\) be a \(\sigma\)--algebra on \(X\) and let \(\mu:\Sigma\to\targetspace\) be a vector measure.
For every \(\lambda \in \reals\) and every \(A\in\Sigma\),
  \begin{equation*}   
    \variation{\targetspace}{\lambda \mu}(A)
    = \abs{\lambda} \,\variation{\targetspace}{\mu}(A).
  \end{equation*}
\end{proposition}

It follows from the previous properties that if the function \(f : X \to \reals\) is bounded and \(\Sigma\)--measurable, then
\begin{align*}
  \biggl\lvert \int_{A} f \dif \mu \biggr\rvert_{\targetspace}
  &\le \int_{A} \abs{f} \dif \variation{\targetspace}{\mu}.
\end{align*}

If \(X\) is a locally compact separable metric space, we endow it with its Borel \(\sigma\)-algebra \(\borel{X}\). This topological setting on \(X\) is common when thinking of
\(C_0(X)^\star\) as the space of finite measures on \(X\), every finite measure being
Radon measures because \(X\) is second-countable and metric. In particular,
whatever the choice of the norm \(\norm{\targetspace}{\cdot}\) on \(\targetspace\),
the variation measure \(\variation{\targetspace}{\mu}\) is a Radon measure.

\begin{proposition}
\label{propositionCharacterizationVariation}
Let \(X\) be a locally compact separable metric space.
If \(\mu : \borel{A} \to V\) is a vector measure and if \(A \subset X\) is open,
then 
  \begin{equation*} 
      \variation{\targetspace}{\mu}(A) 
    = \sup
      {\biggl\{\int_{A} \dualproduct{\psi}{\dif\mu} \st  \psi\in C_c (A, \targetspace') \text{ and } \abs{\psi}_{\targetspace'} \le 1 \biggr\}}. 
  \end{equation*}
\end{proposition}


Here 
\begin{equation*}
  \targetspace' \defeq \bigl\{ \eta \in \targetspace^* \st \norm{\targetspace'}{\eta}< +\infty \bigr\},
\end{equation*}
where
\(\targetspace^*\) is the algebraic dual space of the vector space \(V\) and 
\begin{equation}
\label{eq_dual_norm_primal}
 \norm{\targetspace'}{\eta} \defeq \sup\, \bigl\{ \dualproduct{\eta}{v} \st v \in \targetspace \text{ and } \norm{\targetspace}{v} \le 1\bigr\}.
\end{equation}
Equivalently, \(\eta \in V'\) if and only if for every \(v \in V\), \(\norm{\targetspace}{v} = 0\) implies \(\dualproduct{\eta}{v} = 0\).
By the Hahn--Banach theorem for seminorms (see for example \cite{Brezis}*{Theorem 1.1}), one has 
\begin{equation}
\label{eq_primal_norm_dual}
  \norm{\targetspace}{v} = \sup \,\bigl\{ \dualproduct{\eta}{v} \st \eta \in \targetspace' \text{ and } \abs{v} \le 1\bigr\}.
\end{equation}
Since \(V\) and \(V'\) are finite-dimensional, the suprema in \eqref{eq_dual_norm_primal} and \eqref{eq_primal_norm_dual} are achieved in view of the Weierstrass extreme value theorem.

The proof of \cref{propositionCharacterizationVariation} is a transposition of the proof when \(V\) is Euclidean and \(X\) is separable \cite{Ambrosio_Fusco_Pallara_2000}*{Proposition 1.47}.

\begin{proof}[Proof of \cref{propositionCharacterizationVariation}]
Let \(E _1, \dotsc, E_m \in \borel{X}\) be disjoint subsets of \(A\)
and, in view of \eqref{eq_primal_norm_dual}, let \(\eta_1, \dotsc, \eta_m \in \targetspace'\) be linear forms such that for every \(i \in \{1, \dotsc, m\}\)
\(\norm{\targetspace'}{\eta_i} = 1\) and 
\(\dualproduct{\eta_i}{\mu (E_i)} = \norm{\targetspace}{\mu (E_i)}\).
Since \(\mu\) is a vector measure, \(\variation{\targetspace}{\mu}\) is a finite measure (\cref{proposition_total_variation_countable_additive,proposition_variation_finite}) and thus by Lusin's theorem there exists a function \(\psi \in C (A, \targetspace')\) such that \(\variation{\targetspace}{\mu} (\{x \in A \st \psi (x) \ne \sum_{n = 1}^m \eta_n \chi_{E_n} (x) \}) < \varepsilon\). Without loss of generality we can assume that \(\norm{\targetspace'}{\psi} \le 1\).
Since the set \(A\) is \(\sigma\)--compact and in view Urysohn's lemma for locally compact spaces \cite{Rudin_1987}*{lemma 2.12}, there exists a sequence of functions \((\theta_j)_{j \in \integers}\) in \(C_c (A, [0, 1])\) that converges everywhere to \(1\) in the set \(A\). Therefore, by multiplying \(\psi\) by \(\theta_j\) for \(j \in \integers\) large enough, we can assume that \(\psi \in C_c (A, V')\) and \(\norm{\targetspace'}{\psi} \le 1\) everywhere in \(A\). 
We then have 
  \begin{equation*} 
    \Bigg| \sum_{n = 1}^m \norm{\targetspace}{\mu(E_n)} - \int_A \dualproduct{\psi}{\dif\mu} \Bigg| \leq \varepsilon \variation{\targetspace}{\mu} (A) .
  \end{equation*}
This implies thus that 
\begin{equation*}
 \sum_{n = 1}^m \norm{\targetspace}{\mu(E_n)}\in \closure{\biggl\{\int_{A} \dualproduct{\psi}{\dif\mu} \st  \psi\in C_c (A, \targetspace') \text{ and } \abs{\psi}_{\targetspace'} \le 1 \biggr\}}. 
\end{equation*}

Conversely, if \(\psi \in C_c (A, \targetspace')\) and \(\norm{\targetspace'}{\psi} \le 1\) on \(A\),
then there exist an integer \(m \in \integers\), forms \(\eta_1, \dotsc, \eta_m \in \targetspace'\) and pairwise disjoint sets \(E_1, \dotsc, E_m \in \borel{A}\) such that for each \(n \in \{1, \dotsc, m\}\), \(\norm{\targetspace'}{\eta_n} \le 1\) and \(E_n \subseteq A\) and such that if \(g = \sum_{n = 1}^m \eta_n \chi_{E_n}\), one has \(\norm{\targetspace'}{g - \psi} \le \varepsilon\) everywhere on \(A\). 
We then have 
\begin{equation*}
  \Bigl\vert \int_{A} \dualproduct{g}{\dif \mu} - \int_{A} \dualproduct{\psi}{\dif \mu} \Bigr\vert \le
  \varepsilon \variation{\targetspace}{\mu} (A)
\end{equation*}
and by \cref{definition_mass}
\begin{equation*}
  \int_A \abs{g}{\dif \mu}
  = \sum_{n = 1}^m \dualproduct{\eta_n}{\mu (E_n)}
  \le \sum_{n = 1}^m \norm{\targetspace}{\mu (E_n)}
  \le \variation{\targetspace}{\mu}(A).\qedhere
\end{equation*}
\end{proof}
From now on, even if not stated explicitly, \(X\) will be assumed to be a locally compact
separable metric space endowed with its Borel \(\sigma\)-algebra
\(\Sigma=\mathscr{B}(X)\).

\subsection{Range}
The range of a measure has a definition that has some analogies with the total variation.

\begin{definition}
\label{definition_range}
The \emph{range} of the vector measure \(\mu:\Sigma\to\targetspace\) on a set \(A\in\Sigma\)  is the set
  \begin{equation*} 
      \range[A]{\mu} 
    \defeq 
      \cchull{\big\{ \mu(E): E\in\Sigma \text{ and } E \subset A\big\}} 
    \subseteq 
      \targetspace .
  \end{equation*}
\end{definition}

Here, \(\cchull{H}\) denotes the closed convex hull of a set \(H \subset \targetspace\). The range defined by \cref{definition_range} only depends on the structure of \(V\) as a topological vector space.

If \(\mu = \sum_{i = 1}^k v_i \delta_{x_i}\), with \(x_1, \dotsc, x_k \in X\) distinct points, then \cref{definition_range} gives for every \(A \in \Sigma\)
\begin{equation*}
  \range[A]{\mu} = \Biggl \{ \sum_{i = 1}^d t_i v_i
  \st 0 \le t_i \le \chi_A (x_i)\Biggr\}.
\end{equation*}
If \(X \subset \reals^N\) is a Borel-measurable set and if \(\mu \defeq w \mathcal{L}^N\), where \(w : X \to \targetspace\) is a Borel-measurable function, then 
\begin{multline*}
 \range[A]{\mu}  = \Biggl\{\int_{X} f w \dif \mathcal{L}^N \st f : X \to \reals \text{ is Borel measurable }\\[-1.5em]
 \text{ and } 0 \le f \le \chi_A \text{ in \(X\)}\Biggr\}.
\end{multline*}

It follows immediately from the \cref{definition_range} that the set \(\range[A]{\mu}\) is a closed convex subset of the space \(V\) and that \(0 \in \range[A]{\mu}\). Moreover, \(\range[A]{\mu} = \{0\}\) if and only if \(\mu = 0\).

Since for every \(E \in \Sigma\) such that \(E \subset A\), \(\mu (A \setminus E) = \mu (A) - \mu (E)\), we have 
\(\range[A]{\mu} = \mu (A) - \range[A]{\mu}\), that is, the set \(\range[A]{\mu}\) is symmetric with respect to the vector \(\mu (A)/2 \in V\).

In a finite-dimensional space \(V\), the range is always bounded (see \cref{proposition_variation_finite}).

\begin{proposition}
\label{proposition_range_bounded}
Let \(\Sigma\) be a \(\sigma\)--algebra on \(X\).
If \(\mu:\Sigma\to\targetspace\) is a vector measure, then for every \(A \in \Sigma\), the set \(\range[A]{\mu}\) is bounded.
\end{proposition}

This property fails in general if the domain is not a \(\sigma\)-algebra~\cite{Diestel_Uhl_1977}*{Example~8, Proposition~11}.

\begin{proof}[Proof of \cref{proposition_range_bounded}]
Since \(V\) is finite-dimensional, there exist vectors \(v_1, \dotsc, v_d \in V\) and linear forms \(\eta_1, \dotsc, \eta_d \in \targetspace'\) such that \(d = \dim V\) and for every \(i, j \in \{1, \dotsc, d\}\), \(\dualproduct{\eta_i}{v_j} = \delta_{ij}\), where \(\delta_{ij}\) denotes Kronecker's delta. 
In particular, we have for every \(v \in V\),
\(
 v = \sum_{i= 1}^d v_i \dualproduct{\eta_i}{v}
\),
and thus, for every set \(E \in \Sigma\), 
\(
 \mu (E) = \sum_{i = 1}^d v_i \mu_i (E)
\),
where \(\mu_i (E) = \dualproduct{\eta_i}{\mu (E)}\).
For every \(i \in \{1, \dotsc, d\}\), the set function \(\mu_i\) is a finite signed measure and thus by the Hahn--Jordan decomposition theorem it can be decomposed as 
\(\mu_i =\mu^+_i + \mu^-_i\), where \(\mu_i^+ \ge 0\), \(\mu_i^- \le 0\) are finite Borel  measures~\cite{Conway}*{Appendix~C}.
Hence 
\begin{equation*}
 \range[A]{\mu} 
 \subseteq 
  \biggl\{ \sum_{i = 1}^n \tau_i v_i 
 \st \text{for every \(i \in \{1, \dotsc, d\}\), }
 \mu_i^- (A) \le \tau_i \le \mu_i^+ (A)
 \biggr\},
\end{equation*}
and thus the set \(\range[A]{\mu}\) is bounded.
\end{proof}

The range is monotone with respect to the set, as was the total variation (\cref{proposition_variation_monotone}).

\begin{proposition}[Monotonicity]
Let \(\Sigma\) be a \(\sigma\)--algebra on \(X\), let \(\mu:\Sigma\to\targetspace\) be a vector measure and let \(A, B\in\Sigma\).
If \(A \subseteq B\), then
  \begin{equation*} 
      \range[A]{\mu} 
    \subseteq 
      \range[B]{\mu} 
      .
  \end{equation*}
\end{proposition}
\begin{proof}
This follows directly from \cref{definition_range}.
\end{proof}

Finally the range has a countable additivity property as the total variation (\cref{proposition_total_variation_countable_additive}).

\begin{proposition}[Countable additivity]
\label{proposition_range_countable_additive}
Let \(\Sigma\) be a \(\sigma\)--algebra on \(X\), let \(\mu:\Sigma\to\targetspace\) be a vector measure and let \((A_n)_{n \in \integers}\) be a sequence of disjoint sets in \(\Sigma\).
Then 
\begin{equation*} 
  \range[\textstyle \bigcup_{n\in\integers}A_n]{\mu}
  = \closure{\textstyle \sum_{n\in\integers}\range[A_n]{\mu}} .
\end{equation*}
\end{proposition}

Here \(\sum_{n\in\integers}\range[A_n]{\mu}\) denotes the set of \emph{finite sums} of elements in the sets \(\range[A_n]{\mu}\).

Since the range is compact (\cref{proposition_range_bounded}), \cref{proposition_range_countable_additive} implies that if \(A_1, \dotsc, A_k \in \Sigma\) are disjoint, we have 
\begin{equation*}
 \range[\textstyle \bigcup_{n = 1}^kA_n]{\mu}  
 = \textstyle \sum_{n = 1}^k \range[A_n]{\mu}.
\end{equation*}
If we drop in \cref{proposition_range_countable_additive} the assumption that the sets \((A_n)_{n \in \integers}\) are disjoint we obtain
\begin{equation*}
 \range[\textstyle \bigcup_{n\in\integers}A_n]{\mu}
  \subseteq \closure{\textstyle \sum_{n\in\integers}\range[A_n]{\mu}} .
\end{equation*}

\begin{proof}[Proof of \cref{proposition_range_countable_additive}]
We first assume that \(v \in \sum_{n\in\integers}\range[A_n]{\mu}\). 
Then there exists \(k \in \integers\) and \(v_0, \dotsc, v_k\) such that for each \(n \in \{0, \dotsc, k\}\), \(v_n \in \range[A_n]{\mu}\) and \(v = \sum_{i = 0}^k v_n\). 
By definition of the range, for every \(n \in \{0, \dotsc, k\}\), \(v_n = \lim_{\ell \to \infty} v_n^\ell\), where 
\begin{equation*}
 v_n^\ell \in \chull {\big\{ \mu(E): E\in\Sigma \text{ and } E \subset A_n\big\}} 
    \subseteq 
      \targetspace .
\end{equation*}
Hence we have by additivity of the measure,
\begin{equation*}
 v^\ell \defeq \sum_{n = 0}^k v_n^\ell
 \in \chull {\biggl\{ \mu(E): E\in\Sigma \text{ and } E \subset \bigcup_{i = 0}^n A_n\biggr\}}
\end{equation*}
and, by closedness of the range, \(v \in \range[\bigcup_{n\in\integers}A_n]{\mu}\). We have thus proved that \( \sum_{n\in\integers}\range[A_n]{\mu} \subseteq \range[\bigcup_{n\in\integers}A_n]{\mu}\). Since the range is a closed set by definition, we deduce that \( \closure{\sum_{n\in\integers}\range[A_n]{\mu}} \subseteq \range[\bigcup_{n\in\integers}A_n]{\mu}\)

Conversely, we observe that if \(E \subset \bigcup_{n \in \integers} A_n\), then by countable additivity of the measure \(\mu\), we have
\(
  \mu (E)
  = 
    \sum_{n\in\integers}\mu(E\cap A_n) \in \sum_{n\in\integers}\range[A_n]{\mu}\).
It then follows that \( \range[\textstyle \bigcup_{n\in\integers}A_n]{\mu}
  \subseteq \closure{\textstyle \sum_{n\in\integers}\range[A_n]{\mu}}\).
\end{proof}

The range operation is also compatible with the linear structure of vector measures. First it is sublinear with respect to the sum of measures, as was the total variation \cref{proposition_total_variation_triangle}.

\begin{proposition}[Triangle inequality]
\label{propositionSublinearity}
Let \(\Sigma\) be a \(\sigma\)--algebra on \(X\) and let \(\mu:\Sigma\to\targetspace\) be a vector measure.
For every \(A\in\Sigma\),
  \begin{equation*} 
      \range[A]{(\mu + \nu)}
    \subseteq 
      \range[A]{\mu}+\range[A]{\nu} .
  \end{equation*}
\end{proposition}

The range is not linear. Indeed, in general one has 
\( \range[A]{(\mu - \mu)} = \{0\} \ne \range[A]{\mu}  + \range[A] {(- \mu)}\).

\begin{proof}[Proof of \cref{propositionSublinearity}]
One has for every \(E \in \Sigma\) such that \(E \subset A\),
\(
 \mu (E) \in \range[A]{\mu}+\range[A]{\nu}\).
Since the range is convex and bounded (\cref{proposition_range_bounded}), the set  \(\range[A]{\mu}+\range[A]{\nu}\) is closed and convex and the conclusion follows.
\end{proof}

The range operation is also stable under linear mappings in the target.

\begin{proposition}[Linearity of the range]
\label{proposition_range_linear}Let \(\Sigma\) be a \(\sigma\)--algebra on \(X\), let \(V, W\) be finite-dimensional spaces and let \(\mu:\Sigma\to\targetspace\) be a vector measure.
For every \(T
\in \operatorname{Lin} (V, W)\),
  \begin{equation*} 
    \range{ (T \circ \mu) } = T \bigl(\range{\mu}\bigr).
  \end{equation*}
\end{proposition}

In particular, if \(\lambda\in\reals\) and  \(A\in\Sigma\), we have
  \begin{equation}
    \range[A]{(\lambda\,\mu)} = \lambda\range[A]{\mu} ,
  \end{equation}
  which is the counterpart of \cref{proposition_total_variation_homogeneous}.

\begin{proof}[Proof of \cref{proposition_range_linear}]
If \(E \in \Sigma\) and if \(E \subseteq A\), then we have 
\( (T \circ \mu) (E) = T (\mu (E))\).
It follows then that 
\begin{multline*}
   \cchull {(\{T \circ \mu (E) \st E \in \Sigma \text{ and } E \subseteq A \})}\\
= \closure{T (\chull {(\{\mu (E) \st E \in \Sigma \text{ and } E \subseteq A \}}})
\end{multline*}
We observe now that if the set \(H \subset \targetspace\) is convex and bounded, then by continuity and compactness in the finite-dimensional space \(\targetspace\), \(T (\closure{H}) =\closure{T (H)}\).
\end{proof}

The topological structure on the space \(X\) can be used to characterize the range by duality on test functions.

\begin{proposition}
\label{proposition_range_duality}
If \(X\) is a locally compact separable metric space and if \(\mu : \borel{X} \to V\) is a vector Borel measure and if the set \(A \subset X\) is open,
then 
  \begin{equation*} 
      \range[A]{\mu} 
    = 
      \closure{\biggl\{\int_{A} \varphi \dif\mu  \st  \varphi\in C_c (A, [0,1]) \biggr\}}. 
  \end{equation*}
\end{proposition}

\Cref{proposition_range_duality} is the counterpart of \cref{propositionCharacterizationVariation} for the range. 

\begin{proof}[Proof of \cref{proposition_range_duality}]
Let \(E \in \borel{X}\). Since \(\mu\) is a vector measure, \(\mass{\targetspace}{\mu}\) is a finite measure (\cref{proposition_total_variation_countable_additive,proposition_variation_finite}) and thus by Lusin's theorem there exists a function \(\varphi \in C (X)\) such that \(\variation{\targetspace}{\mu} (\{x \in X \st \varphi (x) \ne \chi_E (x) \}) \le \varepsilon)\).
Since the set \(A\) is open and \(\sigma\)--compact, we can assume without loss of generality that \(\varphi \in C_c (A, [0, 1])\). We then have 
  \begin{equation*} 
    \Bigg| \mu(E) - \int_X\varphi\dif\mu \Bigg| \leq \varepsilon \variation{\targetspace}{\mu} (A) .
  \end{equation*}
This implies thus that 
\begin{equation*}
 \mu (E) \in  \closure{\big\{\action{\mu}{\varphi} : \varphi\in C_c (X,[0,1]) \big\}}. 
\end{equation*}

Conversely, if \(\varphi \in C_c (A)\), there exists a Borel-measurable function \(f : A \to [0, 1]\) that takes only finitely may values and such that \(\abs{f - \varphi} \le \varepsilon\) on \(A\). We then have 
\begin{equation*}
  \Bigl\vert \int_{A} f \dif \mu - \int_{A} \varphi \dif \mu\Bigr\vert \le
  \varepsilon \variation{\targetspace}{\mu}
\end{equation*}
and 
\begin{equation*}
  \int_A f \dif \mu \in \chull{ \{\mu (E) \st E \in \borel{X} \text{ and } E \subset X\}},
\end{equation*}
so that by \cref{definition_range}
\begin{equation*}
 \int_{A} \varphi \dif \mu \in \range[A]{\mu}.\qedhere
\end{equation*}
\end{proof}

\subsection{Links between the range and the total variation }
In the one-dimensional case, there is a straightforward relationship between the range and the mass.

\begin{proposition}
\label{proposition_range_1d}
Assume that \(\dim (V) = 1\). If \(\mu : \Sigma \to \targetspace\) is a vector measure and \(A \in \Sigma\), then 
\begin{equation*}
  \variation{\targetspace}{\mu}(A) = \diam (\range[A]{\mu}).
\end{equation*}
\end{proposition}
\begin{proof}
Without loss of generality, we can assume that \(\targetspace = \reals\).
Let \(\mu = \mu^+ + \mu^-\) be the Hahn--Jordan decomposition (see  e.g. \cite{Conway}*{Appendix~C}) into measures \(\mu^+ \ge 0\) and \(\mu^- \le 0\).
One has then 
\begin{equation*}
  \variation{\reals}{\mu}(A) = \mu^+ (A) - \mu^- (A) 
\end{equation*}
and 
\begin{equation*}
 \range[A]{\mu} = [\mu^- (A), \mu^+ (A)],
\end{equation*}
and the conclusion follows.
\end{proof}

In higher dimensions, the mass and the range are still related by a quantitative bound.

\begin{proposition}[Comparison between total variation and range]
\label{proposition_boundedness}
Let  \((V, \norm{\targetspace}{\cdot})\) be a seminormed space.
There exists a constant \(C>0\) such that for every \(\Sigma\)--algebra on the space 
\(X\) and every vector measure \(\mu : \Sigma \to \targetspace\) and every \(A \in \Sigma\),
we have 
\begin{equation*}
  \range[A]{\mu} \subset \closure{B_\targetspace (0, \variation{\targetspace}{\mu} (A))}
\end{equation*}
and 
\begin{equation*}
    \variation{\targetspace}{\mu} (A)
  \le 
  C \sup \,\bigl\{ \norm{\targetspace}{v} \st v \in \range[A]{\mu}\bigr\}.    
\end{equation*}
\end{proposition}

Diestel and Uhl have a similar statement for the \emph{semivariation} of a vector measure in a Banach space \cite{Diestel_Uhl_1977}*{Proposition~I.1.11}.

\begin{proof}[Proof of \cref{proposition_boundedness}]
\resetconstant
For the first part of the proposition, by \cref{definition_mass}, for every set \(E\in\Sigma\) such that \(E \subseteq A\) we have
  \begin{equation*} \norm{\targetspace}{\mu(E)} \leq \norm{\targetspace}{\mu(E)} + \norm{\targetspace}{\mu(A \setminus E)}
    \leq \variation{\targetspace}{\mu} (A),\end{equation*}
    and the assertion follows from \cref{definition_range}.

For the second part, let \(\eta_1, \dotsc, \eta_d\) be a basis of \(\targetspace'\). By \eqref{eq_primal_norm_dual}, there exists a constant \(\Cl{cst_Au9aDah2il} > 0\) such that for every \(v \in \targetspace\),
\begin{equation*}
    \norm{\targetspace}{v}
  \le
    \Cr{cst_Au9aDah2il}
    \sum_{i = 1}^d
      \abs{\dualproduct{\eta_i}{v}}
      .
\end{equation*}
It follows then that if the sets \(E_1, \dotsc, E_\ell \in \Sigma\) are pairwise disjoint and if \(E_1, \dotsc, E_\ell \subset A\), we have
\begin{equation*}
    \sum_{n = 1}^\ell
      \norm{\targetspace}{\mu (E_i)}
  \le 
    \Cr{cst_Au9aDah2il}
    \sum_{i = 1}^d
    \sum_{n = 1}^\ell
      \abs{\dualproduct{\eta_i}{\mu (E_i)}}
  \le     
    \Cr{cst_Au9aDah2il}
    \sum_{i = 1}^d
    \variation{\targetspace}{\dualproduct{\eta_i}{\mu (E_i)}} (A) ,
\end{equation*}
and using \cref{proposition_range_1d} we infer
	\begin{equation*}
    \sum_{n = 1}^\ell
      \norm{\targetspace}{\mu (E_i)}
     \le  \C 
     \sup \,\bigl\{ \norm{\targetspace}{v} \st v \in \range[A]{\mu}\bigr\} .\qedhere
     \end{equation*}
\end{proof}

\Cref{proposition_boundedness} implies the continuity of the range with respect to the seminorm induced by the total variation.

\begin{proposition}
\label{proposition_strong_continuity_range}
If \(\mu, \nu : \Sigma \to \targetspace\) and \(A \in \Sigma\),
then  
\begin{equation*}
    \hausdist (\range[A]{\nu},  \range[A]{\mu})
  \le  
    \variation{\targetspace}{\nu - \mu}(A).
\end{equation*}
\end{proposition}

Here \( \hausdist (\range[A]{\nu},  \range[A]{\mu})\) denotes the Hausdorff distance between non-empty compact subsets of \(V\): if \(K_1, K_2 \subset V\) are non-empty compact sets,
\begin{equation*}
  \hausdist (K_1, K_2)
  \defeq
  \inf\, 
  \bigl\{ 
      \varepsilon > 0 
    \st 
	  K_1 
	\subseteq 
	  K_2 + \closure{B (0, \varepsilon)} 
      \text{ and }
	  K_2 
	\subseteq 
	  K_1 + \closure{B (0, \varepsilon)}
  \bigr\}.
\end{equation*}

\Cref{proposition_strong_continuity_range} implies in particular that if the sequence \((\mu_n)_{n \in \integers}\) converges in total variation to \(\mu\) on \(A\), then the sequence of sets \((\range[A]{\mu_n} )_{n\in\integers}\) converges to \(\range{\mu}\) in Hausdorff distance.

\begin{proof}[Proof of \cref{proposition_strong_continuity_range}]
By the triangle inequality for the range (\cref{propositionSublinearity}), we have 
\begin{align*}
 \range{\nu} &\subseteq \range{(\nu - \mu)} + \range{\mu}.
\end{align*}
By \cref{proposition_boundedness}, we have 
\begin{align*}
  \range{(\nu - \mu)} &\subseteq \closure{B_\targetspace (0, \variation{\targetspace}{\nu - \mu})}.
\end{align*}
We conclude then that 
\begin{equation*}
 \range{\nu} \subseteq \closure{B_\targetspace (0, \variation{\targetspace}{\nu - \mu})} + \range{\mu}.
\end{equation*}
By symmetry, we also have 
\begin{equation*}
 \range{\mu} \subseteq \closure{B_\targetspace (0, \variation{\targetspace}{\nu - \mu})} + \range{\nu},
\end{equation*}
and the conclusion follows from the definition of Hausdorff distance.
\end{proof}

A second property is the fact that the range determines univocally the image: if two measures share the same range on a set \(A\), then their total variation on that set coincides~\cite{Rodriguez_Piazza}.

\begin{proposition}[Rodr\'\i guez-Piazza]
\label{proposition_range_univ}
Let \(\mu, \nu : \Sigma \to \targetspace\) be vector measures. If \(\range[X]{\mu} = \range[X]{\nu}\), then \(\mass{\targetspace}{\mu} = \mass{\targetspace}{\nu}\).
\end{proposition}

The next lemma highlights the key property used by Rodr\'\i guez-Piazza to prove 
\cref{proposition_range_univ}. It is an adaptation, in terms of total masses, of a representation principle for the norm
by finite combination of positive values of linear functionals~\cite{Rodriguez_Piazza}*{Lemma~2}.

\begin{lemma}\label{polytopeApproximation}
Let \(\norm{\targetspace}{\cdot}\) be a seminorm on \(\targetspace\). 
For each \(\varepsilon>0\), there exists
\(\eta_1,\dots,\eta_m\in \targetspace'\) and \(\alpha_1,\dots,\alpha_m\in\{-1, 1\}\) such that,  for every \(\Sigma\)--algebra on the space 
\(X\) and every vector measure \(\mu : \Sigma \to \targetspace\), we have
  \begin{equation*} \Bigg| \variation{\targetspace}{\mu} (A)
    - \sum^m_{i=1}\alpha_i \variation{\reals}{\dualproduct{\eta_i}{\mu}} (A) 
  \Bigg| \leq \varepsilon\variation{\targetspace}{\mu} (A).\end{equation*}
\end{lemma}

The proof follows the strategy of \cite{Rodriguez_Piazza}*{Theorem 3}.

\begin{proof}[Proof of \cref{polytopeApproximation}]
Fix \(\varepsilon>0\). 
There exist forms \(\eta_1,\dots,\eta_m\in \targetspace'\) and \(\alpha_1,\dots,\alpha_m\in\{-1, 1\}\)  such that for each \(v \in\targetspace\), we have
  \begin{equation*} \Bigg| \norm{\targetspace}{v}
    -  \sum^m_{i=1} \alpha_i |\dualproduct{\eta_i}{v}| 
  \Bigg| \leq \varepsilon\norm{\targetspace}{v} 
  \end{equation*}
(see \citelist{\cite{Choquet}*{p. 53}\cite{Bolker_1969}*{Theorem 2.8}\cite{Rickert_1967_a}*{Lemma 1}}.) Let \(\delta > 0\).
By definition of total mass there exists a finite partition \(\mathcal{P}_* \subset \Sigma\) of \(A\) such that 
\begin{equation*}
  \sum_{E \in \mathcal{P}_*} \norm{\targetspace}{\mu (E)}
  \ge \variation{\targetspace}{\mu}(A)  - \delta.
\end{equation*}
and for every \(i \in \{1, \dotsc, m\}\), there exists 
 a finite partition \(\mathcal{P}_i \subset \Sigma\) of \(A\) such that 
\begin{equation*}
  \sum_{E \in \mathcal{P}_i} \abs{\dualproduct{\eta_i}{\mu (E)}}
  \ge \variation{\reals}{\dualproduct{\eta_i}{\mu}}(A)  -\delta .
\end{equation*}
If we now take \(\mathcal{P}\) to be a common refinement the partitions of \(\mathcal{P}_*, \mathcal{P}_1, \dotsc, \mathcal{P}_m\), we have 
\begin{equation*}
  \sum_{E \in \mathcal{P}} \norm{\targetspace}{\mu (E)} \ge
  \variation{\targetspace}{\mu} (A) 
  -
  \delta
\end{equation*}
and 
\begin{equation*}
  \sum_{E \in \mathcal{P}} \abs{\dualproduct{\eta_i}{\mu (E)}}
  \ge 
  \variation{\reals}{\dualproduct{\eta_i}{\mu}} (A)- \delta .
\end{equation*}
It follows then that 
\begin{equation*}
  \begin{split}
 \Bigg| \variation{\targetspace}{\mu}(A)
    - \sum^m_{i=1}\alpha_i\variation{\reals}{\dualproduct{\eta_i}{\mu}}(A)
  \Bigg| 
  &\leq 
    \varepsilon \sum_{E \in \mathcal{P}} \norm{\targetspace}{\mu (E)}
    +  \biggl\lvert \sum_{E \in \mathcal{P}} \norm{\targetspace}{\mu (E)} - \variation{\targetspace}{\mu} (A)\biggr\vert
    \\
    & \hspace{2.5em} + \sum_{i = 1}^m |\alpha_i|\biggl\lvert  \sum_{E \in \mathcal{P}} \abs{\dualproduct{\eta_i}{\mu (E)}}
      - \variation{\reals}{\dualproduct{\eta_i}{\mu}} (A)\biggr\rvert\\
& \leq \varepsilon\variation {\targetspace}{\mu}(A)
+ \bigg( 1 + \sum^m_{i=1}|\alpha_i|\bigg)\delta.
\end{split}
\end{equation*}
Since \(\delta > 0\) is arbitrary, the conclusion follows.
\end{proof}

Rodr\'\i guez-Piazza's \cref{proposition_range_univ} can be deduced immediately from \cref{polytopeApproximation}.

\begin{proof}[Proof of \cref{proposition_range_univ}]
Let \(\varepsilon > 0\) and let then \(\eta_1, \dotsc, \eta_m \in \targetspace'\) and \(\alpha_1,\dots,\alpha_m\in\{-1, 1\}\) be given by \cref{polytopeApproximation}.
Since \(\range[A]\mu = \range[A] \nu\), we have for every \(i \in \{1, \dotsc, m\}\), in view of \cref{proposition_range_1d},
\begin{equation*}\variation{\reals}{\dualproduct{\eta_i}{\mu}}(A) = \diam (\range[A]{\dualproduct{\eta_i}{\mu}}) = \diam (\range[A]{\dualproduct{\eta_i}{\nu}})
= \variation{\reals}{\dualproduct{\eta_i}{\nu}}(A),
\end{equation*}
and therefore by \cref{polytopeApproximation},
\begin{equation*}
 \bigabs{\variation{\targetspace}{\nu}(A) - \variation{\targetspace}{\mu}(A)}
 \le 
 \varepsilon 
 \,
 \bigl(\variation{\targetspace}{\nu}(A) + \variation{\targetspace}{\mu}(A)\bigr);
\end{equation*}
since \(\varepsilon > 0\) is arbitrary, the conclusion follows.
\end{proof}

\subsection{Wide convergence}

The notion of wide convergence of measures is classically defined by testing the measure against compactly supported functions.

\begin{definition}
\label{definition_wide_convergence}
Let \(X\) be a locally compact separable metric space. 
A sequence of vector measures \((\mu_n)_{n \in \integers}\) from \(X\) to \(\targetspace\) converges widely to some vector measure \(\mu\) whenever for every \(\varphi \in C_c (X)\),
\begin{equation*}
    \lim_{n \to \infty}
      \int_{X} \varphi \dif \mu_n 
  = 
    \int_{X} \varphi \dif \mu.
\end{equation*}
\end{definition}

The wide convergence is compatible with the linear structure of \(V\). If \(T \in \operatorname{Lin} (V, W)\) and if the sequence \((\mu_n)_{n \in \integers}\) converges widely to \(\mu\), then \((T \circ \mu_n)_{n \in \integers}\) converges widely to \(T \circ \mu\).

A classical result is the lower semicontinuity of the mass under weak convergence.

\begin{proposition}
\label{proposition_variation_wlsc}
Let \(X\) be a locally compact separable metric space and let \((\targetspace, \norm{\targetspace}{\cdot})\) be a seminormed finite-dimensional space. 
If \((\mu_n)_{n\in\integers}\) is a sequence of
vector measures from \(X\) to \(V\) widely converging to some vector measure \(\mu : \borel{X} \to V\), then 
  \begin{equation*} 
      \variation{\targetspace}{\mu}(X) 
    \leq 
      \liminf_{n\to\infty}\variation{\targetspace}{\mu_n}(X) .
  \end{equation*}
\end{proposition}
\begin{proof}
Let \(\psi \in C_c (X,\targetspace')\) and assume that \(\norm{\targetspace'}{\psi} \le 1\) everywhere in \(X\).
By \cref{propositionCharacterizationVariation} for every \(n \in \integers\),
\begin{equation*}
 \int_{X} \dualproduct{\psi}{\dif \mu_n} 
 \le 
 \variation{\targetspace}{\mu_n} ,
\end{equation*}
and thus, in view of \cref{definition_wide_convergence}
\begin{equation*}
 \int_{X} \dualproduct{\psi}{\dif \mu_n}
 = \lim_{n \to \infty} \int_{X} \dualproduct{\psi}{\dif \mu_n} 
 \le \liminf_{n \to \infty} \variation{\targetspace}{\mu_n}.
\end{equation*}
We conclude by \cref{propositionCharacterizationVariation}.
\end{proof}

In order to have an analogue for the range, we define the Kuratoswki lower limit of a sequence of set. 

\begin{definition}
\label{definition_Kuratowski_lower_limit}
The Kuratowski lower limit of a sequence \((C_n)_{n \in \integers}\) of subsets of \(\targetspace\) is defined as 
\begin{multline*}
  \liminf_{n \to \infty} C_n 
  \defeq
  \Bigl\{ \lim_{n \to \infty} x_n \st \text{the sequence } (x_n)_{n \in \integers} \text{ converges in }V\\[-.7em]
  \text{ and for every \(n \in \integers\), \(x_n \in C_n\)}\Bigr\}.
\end{multline*}
\end{definition}

By a diagonal argument, the Kuratowski lower limit \( \liminf_{n \to \infty} C_n \) is always a closed set.
Moreover, if \(C_n\) is convex for every \(n \in \integers\), then the Kuratowski lower limit \(\liminf_{n \to \infty} C_n\) is convex. Similarly, if for every \(n \in \integers\), \(0 \in C_n\), then \(0 \in \liminf_{n \to \infty} C_n\).

\begin{proposition}[Weak lower-semi continuity of the range]
\label{proposition_wlsc_range}
If \((\mu_n)_{n\in\integers}\) is a sequence of vector measures from \(X\) to \(V\) that converges widely to some vector measure \(\mu : \borel{X} \to V\), then 
\begin{equation*} 
  \range{\mu} \subseteq \liminf_{n\to\infty}\range{\mu_n} . 
\end{equation*}
Moreover, for every \(\varepsilon > 0\), if \(n \in \integers\) is large enough
\begin{equation*}
\range{\mu} \subseteq \range{\mu_n} + \closure{B (0, \varepsilon)}. 
\end{equation*}
\end{proposition}

In general, the inclusion is strict. Indeed if the space \(X\) is not discrete and \(\targetspace \ne \{0\}\), there exists a point \(x \in X\) and a sequence \((x_n)_{n \in \integers}\) in \(X \setminus \{x\}\) converging to \(x\). If \(v \in \targetspace \setminus\{0\}\), the sequence \((v\delta_x - v \delta_{x_n})_{n \in \integers}\) converges widely to \(0\), but the range of any of its terms is \([-v, v]\) while the range of the limit is \(0\), so that the inclusion is strict.

\begin{proof}[Proof of \cref{proposition_wlsc_range}]
For every \(\varphi \in C_c (X,[0, 1])\), we observe that for every \(n \in \integers\), in view of \cref{proposition_range_duality}
\begin{equation*}
\lim_{n \to \infty} \int_{X} \varphi \dif \mu_n \in \range{\mu_n}
\end{equation*}
Hence in view of \cref{definition_wide_convergence,definition_Kuratowski_lower_limit}, we have
\begin{equation*}
 \int_{X} \varphi \dif \mu
 = \lim_{n \to \infty} \int_{X} \varphi \dif \mu_n
 \in \liminf_{n \to \infty} \range{\mu_n}.
\end{equation*}
Since the Kuratowski lower limit \(\liminf_{n\to\infty}\range{\mu_n}\) is closed and convex, the conclusion follows from \cref{proposition_range_duality}.

For the additional assertion, let \(\varepsilon > 0\). Since the set \(\range{\mu}\) is compact, there exists a finite set \(F \subseteq \range{\mu}\) such that 
\begin{equation}
\label{eq_Eeyieduw3B}
\range{\mu} \subseteq F + \closure{B (0, \varepsilon/2)}.
\end{equation}
Since the set \(F\) is finite, by the first part of the proof, when \(n \in \integers\) is large enough
\begin{equation}
\label{eq_eeZoo1Ahth}
 F \subset \range{\mu_n} + \closure{B (0, \varepsilon/2)}
\end{equation}
and thus by \eqref{eq_Eeyieduw3B} and \eqref{eq_eeZoo1Ahth} we have 
\begin{equation}
\label{eq_ieX0Cheiph}
 \range{\mu} \subseteq \range{\mu_n} + \closure{B (0, \varepsilon)}. \qedhere
\end{equation}
\end{proof}

\section{Convergence of the total variation and of the range}

We consider in this section the relationship, for widely converging sequences between convergence of the range and convergence of the total variation.

\subsection{Range convergence implies mass convergence}

If a sequence of vector measures converges widely and if its ranges converge, then its total variations converge.

\begin{proof}[Proof of \cref{theorem_range_implies_mass}]
We fix \(\varepsilon \in (0, 1)\) and we let \(\eta_1, \dotsc, \eta_m \in \targetspace'\) and \(\alpha_1,\dots,\alpha_m\in\reals\)
be given by \cref{polytopeApproximation}, so that 
\begin{multline*}
 (1 - \varepsilon) \variation{\targetspace}{\mu_n}(X)
 \le (1 + \varepsilon) \variation{\targetspace}{\mu}(X)
 + \sum_{i = 1}^m 
     \bigabs{\alpha_i\variation{\targetspace}{\dualproduct{\eta_i}{\mu}}(X) - \alpha_i\variation{\targetspace}{\dualproduct{\eta_i}{\mu_n}}(X)}.
\end{multline*}
We observe that for every \(n \in \integers\) and every \(i \in \{1, \dotsc, m\}\),
\begin{multline*}
 \bigabs{\diam (\dualproduct{\eta_i}{\range{\mu_n}}) - \diam (\dualproduct{\eta_i}{\range{\mu_n}})}\\
 \le 2 \,\norm{\targetspace'}{\eta_i}\, \hausdist (\range{\mu_n}, \range{\mu}).
\end{multline*}
By \cref{proposition_variation_wlsc} and by \cref{proposition_range_1d}, we have, as \(n \to \infty\),
\begin{multline*}                                                           
\variation{\targetspace}{\dualproduct{\eta_i}{\mu_n}}(X)
= \diam (\dualproduct{\eta_i}{\range{\mu_n}})\\
\to  \diam (\dualproduct{\eta_i}{\range{\mu}})= \variation{\targetspace}{\dualproduct{\eta_i}{\mu}}(X).                                                                        \end{multline*}
If \(\varepsilon < 1\), it follows that 
\begin{equation*}
 \limsup_{n \to \infty} \variation{\targetspace}{\mu_n} (X)
 \le 
 \frac{1 + \varepsilon}{1 - \varepsilon}
  \variation{\targetspace}{\mu} (X).
\end{equation*}
The conclusion follows by letting \(\varepsilon \to 0\).
\end{proof}

\subsection{One-dimensional case}
When \(\dim \targetspace = 1\), the convergence of the range is equivalent to the convergence of the total variation.

\begin{proposition}
\label{proposition_equivalence_1d}
Let \(X\) be a locally compact separable metric space and let \(V\) be a one-dimensional space. 
Let \((\mu_n)_{n \in \integers}\) be sequence from \(X\) to \(V\) converging widely to some vector measure \(\mu\).
If 
\begin{equation*}
 \limsup_{n \to \infty} \mass{\reals}{\mu_n}
 \le \mass{\reals}{\mu},
\end{equation*}
then the sequence \(\bigl(\range{\mu_n}\bigr)_{n\in\integers}\) converges to \(\range{\mu}\) in Hausdorff distance.
\end{proposition}
\begin{proof}
Let \(\varepsilon > 0\). By \cref{proposition_wlsc_range}, 
if \(n \in \integers\) is large enough, we have
\begin{equation}
\label{eq_phikaiNie2}
\range{\mu}
\subseteq \range{\mu_n} 
 + \closure{B (0, \varepsilon)}.
\end{equation}
It follows then that 
\begin{multline}
\label{eq_ieVeiSh3oo}
 \range{\mu_n}
 \subseteq \range{\mu}\\
 +  \closure{B \bigl(0, \bigl(\varepsilon + \diam (\range{\mu_n}) - \diam (\range{\mu})\bigr)_+\bigr)}\;.
\end{multline}
The conclusion follows from \eqref{eq_phikaiNie2} and \eqref{eq_ieVeiSh3oo} since \(\varepsilon > 0\) is arbitrary.
\end{proof}

A corollary of the one dimensional case described above in \cref{proposition_equivalence_1d} is the
following equivalence criterion for range convergence, in terms
of masses of projection measures:

\begin{corollary}
\label{lemma_pointwise_to_hausdorff}
Let \((\mu_n)_{n \in \integers}\) be a sequence of vector measures.
The sequence
\((\range{\mu_n})_{n \in \integers}\) converges to \(\range{\mu}\) in Hausdorff distance if, and only if,
for every \(\eta \in \targetspace'\), 
\begin{equation*}
  \limsup_{n \to \infty} \variation{\reals}{\dualproduct{\eta}{\mu_n}}(X)
  \le \variation{\reals}{\dualproduct{\eta}{\mu}}(X) .
  \end{equation*}
\end{corollary}
\begin{proof}
Let \(\varepsilon > 0\). 
For every \(\eta \in  \targetspace'\), the sequence \((\dualproduct{\eta}{\mu_n})_{n \in \integers}\) of signed measures converges widely to the signed measure \(\dualproduct{\eta}{\mu}\).
By \cref{proposition_equivalence_1d}, we have then
\begin{equation}
\label{eq_JaePeeGh5u}
    \lim_{n \to \infty} 
      \hausdist 
	\bigl(\range{\dualproduct{\eta_i}{\mu_n}}, \range{\dualproduct{\eta_i}{\mu}}\bigr) 
  = 
    0.
\end{equation}
Since the set \(\range{\mu} \subset \targetspace\) is compact and convex, there exist \(\eta_1, \dotsc, \eta_m \in \targetspace'\) and \(\delta > 0\) such that
\begin{equation}
\label{eq_ieBoh5IoZu}
 \range{\mu} \subseteq \bigl\{ v \in \targetspace \st \text{ for each \(i \in \{1, \dotsc, m\}\), } \dualproduct{\eta_i}{v} \le 1\bigr\}
\end{equation}
and 
\begin{multline}
\label{eq_ooloshaeN3}
 \bigl\{ v \in \targetspace \st \text{ for each \(i \in \{1, \dotsc, m\}\) } \dualproduct{\eta_i}{v} \le 1 + \delta\bigr\}\\
 \subseteq \range{\mu} + \closure{B (0, \delta)}.
\end{multline}
In particular, if \(n \in \integers\) is large enough, we have by \cref{proposition_range_linear} and by  \eqref{eq_JaePeeGh5u} and \eqref{eq_ieBoh5IoZu},
\begin{equation*}
  \sup \, \dualproduct{\eta_i}{\range{\mu_n}}
  = \sup \range{\,\dualproduct{\eta_i}{\mu_n}}
  \le 1 + \delta,
\end{equation*}
and thus by \eqref{eq_ooloshaeN3}
\begin{equation}
\label{eq_thiu1Iejie}
 \range{\mu_n}
 \subseteq \range{\mu} + \closure{B (0, \varepsilon)}
 .
\end{equation}

We conclude the first implication by combining the inclusion \eqref{eq_thiu1Iejie} with \cref{proposition_wlsc_range}.
The converse claim of the proposition follows from the fact that
range convergence always implies mass convergence, and
	\begin{equation*} \variation{\reals}{\dualproduct{\eta}{\range{\mu}}}
		= \variation{\tilde{V}}{\mu} ,\end{equation*}
where \(\tilde{V}\) is the vector space \(V\) endowed with the
degenerate norm
  \begin{equation*} 
    \norm{\tilde{V}}{x} = |\scalarproduct{x}{\eta}{}| .\qedhere
  \end{equation*}
\end{proof}

\subsection{Higher-dimensional case}

In this section, we prove an equivalence result between
convergence of the range and convergence of the total mass,
when the mass is taken with respect to any strictly convex norm.

A seminorm \(\norm{\targetspace}{\cdot}\) is \emph{strictly convex} whenever if the vectors \(v, w \in V\) are linearly independent, then
\begin{equation*}
 \norm{\targetspace}{v + w} 
 < \norm{\targetspace}{v} + \norm{\targetspace}{w}.
\end{equation*}

As claimed in \cref{proposition_strict_convexity} stated in the introduction,
strict convexity of the norm is necessary to enforce the convergence of ranges.
The proof of~\cref{proposition_strict_convexity} we propose is based
on an straightforward counter-example involving Dirac masses:

\begin{proof}[Proof of \cref{proposition_strict_convexity}]
Since by our assumption the seminorm \(\norm{\targetspace}{\cdot}\) is not strictly convex, there exist two linearly independent vectors \(v, w\in\targetspace\setminus\{0\}\) such that
\(\norm{\targetspace}{v+w} = \norm{\targetspace}{v} + \norm{\targetspace}{w}\).
Since the space \(X\) is not discrete, there exists a point \(x \in X\) and a sequence of points \((x_n)_{n \in \integers}\) in \(X \setminus \{x\}\) that converges to \(x\).
We define for each  
\(n\in\integers\), the vector measure
\begin{equation*} 
    \mu_n 
  \defeq
    v \,\delta_{x} + w\, \delta_{x_n} 
    .
\end{equation*}
The sequence \((\mu_n)_{n\in\integers}\) then widely converges to the measure
\(\mu = (v + w)\,\delta_{x}\)
and we have for each \(n \in \integers\)
  \begin{equation*} 
      \mass{\targetspace}{\mu} 
    = 
      \norm{\targetspace}{v+w} 
    = 
  \norm{\targetspace}{v} 
      + 
  \norm{\targetspace}{w}
  = \mass{\targetspace}{\mu_n} .
  \end{equation*}
On the other hand the range of \(\range{\mu_n} = \chull \{0,v,w,v+w\}\)
 while \(\range{\mu} =  \chull \{0,v+w\}\) corresponds to the diagonal of the parallelogram \(\range (\mu_n)\). 
Since the vectors \(v\) and \(w\) are linearly independent, the sequence of ranges does not converge.
\end{proof}  

In particular, by \cref{proposition_strict_convexity} and \cref{proposition_equivalence_1d} there exists a linear map \(T :\targetspace \to \reals\) such that the sequence \((\variation{\reals}{T \circ \mu_n}(\Omega))_{n \in \integers}\) does not converge to \(\variation{\reals}{T \circ \mu}(\Omega)\). 
That is, in contrast with wide convergence and with convergence of ranges, the convergence of total variations is not invariant under linear changes of variables.

\begin{theorem}\label{theorem_strict_convexity_equivalence}
Assume that \(X\) is a locally compact separable metric space, and let
\((\mu_n)_{n\in\integers}\) be a sequence of vector measures weakly converging
to some vector measure \(\mu\). If \(\norm{\targetspace}{\cdot}\) is a strictly
convex norm on \(\targetspace\) such that
	\begin{equation*} \lim_{n\to\infty}\variation{\targetspace}{\mu_n} (X)
		= \variation{\targetspace}{\mu} (X),\end{equation*}
then the sequence of ranges \((\range{\mu_n})_{n \in \integers}\) converges
to \(\range{\mu}\) in Hausdorff distance. 
\end{theorem}
The proof follows
lines of thought from the work of Reshetnyak~\cite{Reshetnyak}*{Theorem~3},
Since the approach requires some tools from measure theory,
we repeat the proof with explanations on how the tools are used, and
we chain the argument with \cref{lemma_pointwise_to_hausdorff}
to obtain Hausdorff convergence of the ranges.

\begin{proof}[Proof of~\cref{theorem_strict_convexity_equivalence}]
For all \(n\in\integers\), let \(\chi_{\mu_n}\) be the \emph{indicator measure}~\cite{Reshetnyak} of \(\variation{\targetspace}{\mu}\), defined as the unique measure
	\begin{equation*} \chi_{\mu} : \mathscr{B}(X\times \sphere)
		\to [0, +\infty) \end{equation*}
acting on any continuous function \(\varphi\in C(X\times\sphere,\reals)\) as
	\begin{equation*} \int_{X\times\sphere}\varphi(x,e)\dif\chi_{\mu}(x,e)
        = \int_{X}
        \varphi\Bigl(x,\frac{\dif \mu_n}{\dif \variation{\targetspace}{\mu}}(x)\Bigr) \dif \variation{\targetspace}{\mu} (x).\end{equation*}

Since we are working with vector measures,
and because the ranges and the masses remain uniformly bounded for \(n\in\integers\),
we may pass to subsequences and
assume without loss of generality
that the sequence \((\chi_{\mu_n})_{n\in\integers}\) weakly converges
to some measure \(\rho\).
As in~\cite{Reshetnyak}*{Theorem~3}, we are going to prove that
the limit \(\rho\) actually equals \(\chi_{\mu}\), the indicator function of \(\variation{\targetspace}{\mu}\).
To that purpose, we use~\cite{Reshetnyak}*{Theorem~1} to obtain a layerwise decomposition \((\omega,\lambda)\) 
of the indicator measure \(\rho\)~\cite{Reshetnyak}, in the sense that
\(\omega: X\to[0, +\infty)\) is a positive measure and \(\lambda: X\to\mathcal{M}\big(\mathscr{B}(\sphere)\big)\) is a function
such that for every \(\varphi \in C (X \times \sphere)\),
\begin{equation*} \int_{X\times\sphere}\varphi(x,e)\dif\rho(v,e)
= \int_X\int_\sphere\varphi(x,e)\dif \lambda(x,e)\dif\omega(x) .
\end{equation*}
Moreover, for this layerwise decomposition, the function
	\begin{equation*} 
          x\in X \mapsto u(x) = \int_{\sphere}e\dif\lambda(x,e) 
        \end{equation*}
        is \(\mathscr{B}(X)\)-measurable and for every \(A \in \borel{X}\),
	\begin{equation}\label{eqn_identity_RThm1}
		\int_Au(x)\dif\omega(x)
			= \int_A\int_{\sphere}e\dif\lambda(x,e)\dif\omega(x) = \mu(A) .
	\end{equation}
The existence of such decomposition with relation~\eqref{eqn_identity_RThm1}
only uses the weak convergence assumption
\(\mu_n\rightharpoonup\mu\).
By convexity of the norm, we have
	\begin{equation}\label{inequality_resolved_by_identity}
		\norm{\targetspace}{u(x)} \leq \int_{\sphere}\dif\lambda(x,e) \leq 1 .
	\end{equation}
Those inequality follow from~\cite{Reshetnyak}*{Theorem~1}, where it is
proved that \(\lambda(x,\sphere)=1\) for \(\omega\)-almost-all \(x\in X\).
Integrating left sides of the above inequality \eqref{inequality_resolved_by_identity} with respect to the measure \(\omega\)
also yields
	\begin{equation*} \int_X\norm{\targetspace}{u(x)}\dif\omega(x)
		\leq \int_X\int_{\sphere}\dif\lambda(x,e)\dif\omega(e)
		= \int_{X\times\sphere}\dif\rho(x,e) ,\end{equation*}
and by construction of \(\rho\):
	\begin{equation*} \int_{X\times\sphere}\dif\rho(x,e)
		= \lim_{n\to\infty}\int_{X\times\sphere}\dif\chi_{\mu_n}(x,e)
		= \lim_{n\to\infty}\variation{\targetspace}{\mu_n}(X) = \variation{\targetspace}{\mu}(X) .\end{equation*}
Thus
	\begin{equation*} \int_X\norm{\targetspace}{u(x)}\dif\omega(x)
		\leq \variation{\targetspace}{\mu} .\end{equation*}

For \(A\in\mathscr{B}(X)\) an
arbitrary Borel set, we observe that
	\begin{equation*} \int_A\frac{\dif \mu}{\dif \variation{\targetspace}{\mu}} \dif\variation{\targetspace}{\mu}  = \mu(A)
		= \int_A u\dif\omega .\end{equation*}
Considering the Radon-Nikodym derivatives of \(\variation{\targetspace}{\mu}\) and \(\omega\)
with respect to the measure \(\sigma\defeq(\variation{\targetspace}{\mu}+\omega)\), we obtain
	\begin{equation*} \int_A \frac{\dif \mu}{\dif \variation{\targetspace}{\mu}}\ \frac{\dif\variation{\targetspace}{\mu}}{\dif(\omega+\variation{\targetspace}{\mu})} \dif\sigma
		= \int_A u\ \frac{\dif\omega}{\dif(\omega+\variation{\targetspace}{\mu})}
			\dif\sigma ,\end{equation*}
and since the set \(A\in\mathscr{B}(X)\) is arbitrary, we obtain 
\begin{equation}\label{eqn_nikodym_identity}
  u\ \frac{\dif\omega}{\dif(\omega+\variation{\targetspace}{\mu})}
  = \frac{\dif \mu}{\dif \variation{\targetspace}{\mu}}\ \frac{\dif\variation{\targetspace}{\mu}}{\dif(\omega+\variation{\targetspace}{\mu})} 
\end{equation}
\((\variation{\targetspace}{\mu}+\omega)\)-almost-everywhere on \(X\).
Using the fact that \(\omega\) and \(\variation{\targetspace}{\mu}\) are positive measures, 
we take norms on both sides of \eqref{eqn_nikodym_identity} and we integrate both
sides with respect to \(\sigma\) to obtain
	\begin{equation*} \int_X\norm{\targetspace}{u(x)}\dif\omega(x) 
		\int_X \biggnorm{\targetspace}{\frac{\dif \mu}{\dif \variation{\targetspace}{\mu}}(x)}\dif\variation{\targetspace}{\mu}(x)
			= \variation{\targetspace}{\mu}(X) \end{equation*}
and equality holds in the inequality~\eqref{inequality_resolved_by_identity}: we have
\(\norm{\targetspace}{u(x)}=1\) for \(\omega\)-almost-all \(x\in X\).
In particular, \(\omega=\variation{\targetspace}{\mu}\)
and \(u=\frac{\dif \mu}{\dif \variation{\targetspace}{\mu}}\) on \(X\).

The strict convexity of the norm \(\norm{\targetspace}{\cdot}\) may be used
to observe that in fact, \(\lambda(x,\cdot)\) is a Dirac mass at \(u(x)\), for
\(\variation{\targetspace}{\mu}\)-almost-all \(x\in X\). Indeed, this follows from
the fact that
	\begin{equation*} 
          \biggnorm{\targetspace}{\int_{\sphere}e\dif\lambda(x,e)}
		= \int_{\sphere}\norm{\targetspace}{e}\dif\lambda(x,e) ,\end{equation*}
which is possible if and only if \(\lambda(x,\cdot)\) is a multiple of a Dirac mass.
The concentration point should thus be \(u(x)\), and the multiple is \(1\).

Now we fix an arbitrary direction \(\eta\in\sphere\).
Considering again Radon-Nikodym derivatives, it is straightforward to check
the identity
	\begin{equation*} \variation{\reals}{\dualproduct{\mu}{\eta}}(X)
		= \int_X\Big|\Bigdualproduct{\frac{\dif \mu}{\dif \variation{\targetspace}{\mu}}(x)}{\eta}\Big|\dif\variation{\targetspace}{\mu}(x) ,\end{equation*}
or equivalently
	\begin{equation*} \variation{\reals}{\dualproduct{\mu}{\eta}}(X)
		= \int_X\big|\dualproduct{u(x)}{\eta}\big|\dif\omega(x) .\end{equation*}
Using the fact that \(\lambda(x,\cdot)\) is a Dirac mass at \(u(x)\), the inner
absolute value commutes the integral to give
	\begin{equation*} \int_X\Bigl|\Bigdualproduct{u(x)}{\eta}\Bigr|\dif\omega(x)
		= \int_X\int_{\sphere}|\dualproduct{e}{\eta}|\dif\lambda(x,e)\dif\omega(x)
		= \int_{X\times\sphere}|\dualproduct{e}{\eta}|\dif\rho(x,e) .\end{equation*}
Now we always have
	\begin{equation*} \variation{\reals}{\dualproduct{\mu_n}{\eta}}(X)
		= \int_X\Big|\Bigdualproduct{\frac{\dif \mu}{\dif \variation{\targetspace}{\mu}}(x)}{\eta}\Big|\dif\variation{\targetspace}{\mu}(x) = \int_{X\times\sphere}|\dualproduct{e}{\eta}|\dif\chi_{\mu_n}(x,e) ,\end{equation*}
and therefore, by definition of \(\rho\), we obtain
	\begin{equation*} \lim_{n\to\infty}\variation{\reals}{\dualproduct{\mu_n}{\eta}}(X)
		= \variation{\reals}{\dualproduct{\mu}{\eta}}(X) .\end{equation*}
Since this holds for all \(\eta\in\sphere\), we obtain Hausdorff convergence of the range,
by~\cref{lemma_pointwise_to_hausdorff}.
\end{proof}

\Cref{theorem_range_implies_mass} and \cref{theorem_strict_convexity_equivalence}
imply immediately \cref{theorem_Euclidean}.

\section{Zonal representation of masses}

In this section, we develop a geometric approach to manage the
relation between range convergence and mass convergence. Although
the geometric approach will give much less results when
\(\dim(\targetspace)\geq 3\), it provides an interesting geometric point of view
of the relation between mass convergence and strict convergence
when \(\dim(\targetspace)=2\).

\subsection{Zonal representation formula}
We will show that seminorms having a zonal representation have an associated mass that has an integral representation. 

\begin{definition}
  \label{def_zonal_representation}
A locally finite non-negative Borel measure \(\sigma:\borel{\targetspace'}\to
  [0, +\infty]\) is a \emph{zonal representation} of a seminorm \(\norm{\targetspace}{\cdot}\)
 whenever 
\begin{equation*}
  \int_{\targetspace'} \norm{\targetspace'}{\eta} \dif \sigma (\eta) < + \infty,
\end{equation*}
and for every \(v \in \targetspace\),
\begin{equation*}
\label{eq_zonal_representation}
  \norm{\targetspace}{v} 
  = 
  \int_{\targetspace'} \abs{\dualproduct{\eta}{v}} \dif \sigma(\eta).
\end{equation*}
\end{definition}

As an example, we can consider the standard Euclidean
norm in \(\reals^2\) and the uniform measure \(\nu\) on
the unit circle \(\sphere \subset \reals^2\), giving
	\begin{equation*} \int_{\sphere}|\dualproduct{v}{e}|\dif\rho(e)
		= \frac{|v|_2}{2\pi}\int_0^{2\pi}|\cos(\theta)|\dif\theta
		= |v|_2\,\frac{2}{\pi} ,\end{equation*}
and the zonal representation holds with
\(\sigma=\frac{\pi}{2}\rho\).

More generally, every Euclidean norm admits a zonal
representation.
\begin{theorem}[Representation of Euclidean norms]
If \(\norm{\targetspace}{\cdot}\) is a Euclidean semi-norm
on \(\targetspace\), then it admits a zonal
representation.
\end{theorem}
\begin{proof}
Up to a linear isometry and an orthogonal projection, we can assume that \(\targetspace=\reals^d\) and
\(\norm{\targetspace}{\cdot}=\norm{2}{\cdot}\) is the standard Euclidean
norm. Then we have
	\begin{equation*} \norm{2}{v}
		= \alpha\int_{\sphere}|\dualproduct{v}{e}|\dif\sigma(e) ,\end{equation*}
where \(\sigma\) is the uniform surface measure on the
Euclidean sphere \(\sphere\), and \(\alpha=\alpha(d)>0\)
is a normalization constant. The above identity
directly follows from the fact that the uniform
surface measure is the unique (up to multiplication
constants) Borel measure on \(\sphere\) that is stable
under isometries of \(\reals^d\).
\end{proof}

A norm \(\norm{\targetspace}{\cdot}\) has a zonal representation if and only if the associated unit ball in \(\dualspace\) is a \emph{zonoid}, which can be defined as the ranges of vector measure \citelist{\cite{Choquet}\cite{Bolker_1969}*{Theorem 2.1}}.

If \(\targetspace\) is a two-dimensional space, every norm has a zonal representation in the sense of \cref{def_zonal_representation} (\citelist{\cite{Choquet}*{41.2}\cite{Bolker_1969}*{Theorem 5.4}}). This follows from the fact that, in the plane \(\reals^2\), every symmetric compact convex set may be approximated by a symmetric convex polygon with is automatically a zonotope in \(\reals^2\).

For example, the usual Euclidean norm is represented by the uniform
measure on the dual sphere \(\sphere\) of \(V'\) since 
the uniform measure is the only (up to multiplicative constants) invariant measure under rotations
and thus, balls induced by its associated norm are Euclidean balls. This argument is
also valid in higher dimensions.

Another example is given by the \(\ell^1\)-norm on \(\reals^d\) defined for \(v \in \reals^d\) by  
\begin{equation*}
 \norm{\infty}{v} 
 \defeq 
 \sum_{i = 1}^d \abs{v_1},
\end{equation*}
which 
can be represented with \(\sigma = \sum_{i = 1}^d \delta_{\eta_i}\),
where \(\eta_i (v_1, \dotsc, v_d) = v_i\).

\medbreak

In higher dimensions \(\dim V \ge 3\), there are some norms that do not have a zonal representation. 
If \(v = (v_1, \dotsc, v_d) \in \reals^d\), then 
  \begin{equation*} 
    \norm{\infty}{v} = \max \{\abs{v_1}, \dotsc, \abs{v_d}\}.
  \end{equation*}
\emph{has no zonal representation} of the form \cref{def_zonal_representation} by any positive measure \(\sigma\).
Indeed, if such measure \(\sigma\) existed, then its dual ball, which is the \(\ell^1\)--ball, should be a zonotope.
However, the \(1\)--ball is a hyperoctahedron, which has some \(2\)--dimensional that are triangles, and is thus not a zonotope \cite{Bolker_1969}*{Theorem 3.3}.

The above example is not related with the fact that the \(\ell^1\) and \(\ell^\infty\) norms
are not strictly convex. In fact, there are also strictly convex norms that are do not have a zonal representation \cref{def_zonal_representation}. 
Indeed, in \(\reals^3\), the \(p\)--norm does not have a zonal representation, if \(p \ge \frac{\log{3}}{\log \frac{3}{2}}\), and for every \(p > 2\) it is known that the \(p\)--norm does not have a zonal representation is the dimension is large enough \citelist{\cite{Witsenhausen_1973}*{Proposition 2}\cite{Bolker_1969}*{(32)}}.

\begin{remark}
\label{remark_sigma}
If \(\sigma' : \borel{\targetspace'} \to [0, + \infty)\) is a measure such that for every Borel cone \(C \in \borel{\targetspace'}\), \(\sigma'(C) = \sigma(C)\), then for every \(v \in \targetspace\),
\begin{equation*} \int_{\targetspace'} \abs{\dualproduct{\eta}{v}} \dif \sigma'(\eta) = \int_{\targetspace'} \abs{\dualproduct{\eta}{v}} \dif \sigma(\eta)
\end{equation*}
and for every vector measure \(\mu : \Sigma \to \targetspace\),
\begin{equation*}
 \int_{\targetspace'} \variation{\reals}{\dualproduct{\eta}{\mu}} (A) \dif \sigma'(\eta)= \int_{\targetspace'} \variation{\reals}{\dualproduct{\eta}{\mu}} (A) \dif \sigma(\eta).
\end{equation*}
In particular, the measure \(\sigma\) satisfying the assumption of \cref{lemma_mass_representation} is not unique; a canonical representative can be given by requiring sigma to be an even measure supported on the unit sphere of \(\targetspace'\).
\end{remark}

The existence of zonal representation of the norm in \(\targetspace\)
automatically implies a similar representation formula for the mass.
It may be seen as an alternative way of representing the mass, without relying
on layerwise decompositions.

\begin{lemma}[Representation of the mass]
\label{lemma_mass_representation}
If \(\sigma:\borel{\targetspace'}\to\reals^+\) is a zonal representation of \(\norm{\targetspace}{\cdot}\),
then for every \(\sigma\)--algebra \(\Sigma\) on \(X\), every vector measure \(\mu : \Sigma \to \targetspace\) and every \(A \in \Sigma\),
\begin{equation*}
 \variation{\targetspace}{\mu}(A) = \int_{\targetspace'} \variation{\reals}{\dualproduct{\eta}{\mu}} (A) \dif \sigma(\eta).
\end{equation*}
\end{lemma}
\begin{proof}
If \(E_1, \dotsc, E_m \in \Sigma\) are disjoints subsets of \(A\), we have 
\begin{equation*}
  \sum_{i = 1}^m 
    \norm{\targetspace}{\mu (E_i)}
  =
    \int_{V'} \sum_{i = 1}^m \abs{\dualproduct{\eta}{\mu (E_i)}} \dif \sigma (\eta)
  \le \int_{V'} \variation{\reals}{\dualproduct{\eta}{\mu}}(A) \dif \sigma (\eta)
\end{equation*}
and thus by the definition of total variation of a vector measure (\cref{definition_mass})
\begin{equation}
\label{eq_aelohMa1ei}
  \variation{\targetspace}{\mu}(A)
  \le  \int_{V'} \variation{\reals}{\dualproduct{\eta}{\mu}}(A) \dif \sigma (\eta).
\end{equation}

Conversely, in view of \cref{remark_sigma} we can assume that the measure \(\sigma\) is supported in a compact set \(K \subset \targetspace' \setminus  \{0\}\). In particular \(\sigma\) is a finite measure.
Let \(\varepsilon > 0\), there exists \(\eta_1, \dotsc, \eta_m \in K\) such that 
\begin{equation*}
 K \subset \bigcup_{i = 1}^m \closure{B_{\targetspace'}(\eta_i, \varepsilon)}.
\end{equation*}
In particular, if \(\eta \in \closure{B_{\targetspace'}(\eta_i, \varepsilon)}\) and \(E \in \Sigma\), we have 
\begin{equation}
\label{eq_jikaiyah4I}
  \bigabs{ \variation{\reals}{\dualproduct{\eta}{\mu}}(E) - \variation{\reals}{\dualproduct{\eta_i}{\mu}}(E)} \le \varepsilon \variation{\targetspace}{\mu} (E).
\end{equation}
We decompose \(K = \bigcup_{i = 1}^\ell K_i\), with \(K_1, \dotsc, K_\ell \in \borel{\targetspace'}\) pairwise disjoint and \(K_i \subset \closure{B_{\targetspace'} (\eta_i, \varepsilon)}\).
We have then by \eqref{eq_jikaiyah4I}
\begin{equation}
\label{eq_be5Eir3buu}
 \int_{\targetspace'} \variation{\reals}{\dualproduct{\eta}{\mu}}(A) \dif \sigma (\eta)
 \le \sum_{i = 1}^d \sigma (K_i)\, \variation{\reals}{\dualproduct{\eta_i}{\mu}}(A) + \varepsilon \sigma (K) \variation{\targetspace}{\mu} (A).
\end{equation}
Since \(\dualproduct{\eta_i}{\mu}\) is a signed measure, by the Hahn--Jordan decomposition theorem and a common refinement, there exists a partition \(E_1, \dotsc, E_m \in \Sigma\) of \(A\) such that for each \(i \in \{1, \dotsc, \ell\}\),
\begin{equation}
\label{eq_Ohyoi4zeiN}
 \variation{\reals}{\dualproduct{\eta_i}{\mu}}(A)
 = \sum_{n = 1}^m \abs{\dualproduct{\eta_i}{\mu} (E_n)}. 
\end{equation}
We obtain thus by  \eqref{eq_jikaiyah4I} again combined with \eqref{eq_be5Eir3buu} and \eqref{eq_Ohyoi4zeiN} 
\begin{multline*}
 \int_{\targetspace'} \variation{\reals}{\dualproduct{\eta}{\mu}}(A) \dif \sigma (\eta)\\
 \le \sum_{n = 1}^m \int_{\targetspace'} \abs{\dualproduct{\eta}{\mu} (E_n)} \dif \sigma (\eta) + \varepsilon \sigma (K) \Bigl(\variation{\targetspace}{\mu}(A) + \sum_{n  =1}^m \variation{\targetspace}{\mu}(E_n) \Bigr).
\end{multline*}
By our assumption on the seminorm this implies in turn by our assumption and by additivity of the total mass (\cref{proposition_total_variation_countable_additive}) and by the definition of total variation  (\cref{definition_mass}) that 
\begin{equation*}
\begin{split}
 \int_{\targetspace'} \variation{\reals}{\dualproduct{\eta}{\mu}}(A) \dif \sigma (\eta)
 &\le \sum_{n = 1}^m \norm{\targetspace}{\mu (E_n)}
 + 2 \varepsilon \sigma (K) \variation{\targetspace}{\mu}(A)\\
 &\le \variation{\targetspace}{\mu} (A) + 2 \varepsilon \sigma (K) \variation{\targetspace}{\mu}(A).
\end{split}
\end{equation*}
We conclude by letting \(\varepsilon \to 0\) that 
\begin{equation}
\label{eq_Iesa2quahV}
 \int_{\targetspace'} \variation{\reals}{\dualproduct{\eta}{\mu}}(A) \dif \sigma (\eta)
 \le \variation{\targetspace}{\mu} (A).
\end{equation}
The conclusion then follows from \eqref{eq_aelohMa1ei} and \eqref{eq_Iesa2quahV}.
\end{proof}

\subsection{Application to monotonicity and convergence}

When \(\targetspace\) is equipped with an Euclidean norm
the monotony of the mass with respect to range inclusion,
\cref{mass_monotony} stated in the introduction,
is a direct consequence of the zonal representation, \cref{lemma_mass_representation}.

Because of the particular structure of two-dimensional
vector spaces, the proof we give also holds
for any norm on \(\targetspace\) whenever
\(\dim\targetspace=2\).

\begin{proof}[Proof of~\cref{mass_monotony}]
Given two vector measures \(\mu,\nu\) such that
	\begin{equation*} \range{\mu}\subseteq\range{\nu} ,\end{equation*}
we clear have, for all \(\eta\in\sphere\):
	\begin{equation*} \dualproduct{\eta}{\range{\mu}}
		\subseteq \dualproduct{\eta}{\range{\nu}} .\end{equation*}
Since the range commutes with linear maps,
we directly obtain, for all \(\eta\in\sphere\):
	\begin{equation*} \range{\dualproduct{\eta}{\mu}}
		\subseteq \range{\dualproduct{\eta}{\nu}} .\end{equation*}
Since the measures \(\dualproduct{\eta}{\mu}\)
and \(\dualproduct{\eta}{\nu}\) are real-valued, we obtain
	\begin{equation*} |\dualproduct{\eta}{\mu}|
		\leq |\dualproduct{\eta}{\nu}| ,\end{equation*}
for all \(\eta\in\sphere\).
Since \(V\) is Euclidean or \(\dim V = 2\), this implies,
by~\cref{lemma_mass_representation},
\(\variation{\targetspace}{\mu}(X)
		\leq \variation{\targetspace}{\nu}(X)\).
                If the norm is strictly convex, then the inequality is strict.
\end{proof}

Relying on the integral and perimeter formulae obtained above, we also propose an alternative proof of
\cref{theorem_Euclidean}.

We first draw a link between the support of a zonal representation and
strict convexity. The above characterization will turn out to be
particularly efficient in dimension~2.

\begin{definition}
  \label{def_zonal_representation_non-degenerate}
  A \emph{zonal representation} \(\sigma:\borel{\targetspace'}\to
  [0, +\infty]\) of a seminorm \(\norm{\targetspace}{\cdot}\)
  is \emph{non-degenerate} whenever for every non-empty double cone
  \(C \subset \targetspace\), 
  \(\sigma (C) > 0\).
\end{definition}

We recall that the set \(C \subset V'\) is an \emph{open cone} whenever for every \(\lambda \in \reals \setminus \{0\}\) 
and \(x \in \), \(\lambda C = C\).

In dimension~2, the non-degeneracy of the zonal representation is equivalent to strict convexity.
This will turn out to be false in higher dimensions.

\begin{proposition}
\label{proposition_strict_convexity_zonal}
Let \(V\) be a finite-dimensional space.
Assume that \(\sigma\) is a zonal representation of the seminorm \(\norm{\targetspace}{\cdot}\).
  The norm \(\norm{\targetspace}{\cdot}\) is strictly convex if and only if for every linearly independent vectors \(v, w \in\targetspace\), there exists \(\eta\in\support(\sigma)\) such that
 \(\dualproduct{\eta}{v}\) and \(\dualproduct{\eta}{w}\) do not have the same sign.\\ 
In particular, if \(\targetspace\) is a \(2\)-dimensional vector space, then
the norm is strictly convex if, and only if, \(\sigma\) is non-degenerate.
\end{proposition}

\begin{proof}[Proof of \cref{proposition_strict_convexity_zonal}]
We observe that for every \(\eta \in \targetspace'\) and \(v \in \targetspace\), we have
\begin{equation*}
 \bigabs{\dualproduct{\eta}{v}  + \dualproduct{\eta}{w}}
  \ge \bigabs{\dualproduct{\eta}{v}}  + \bigabs{\dualproduct{\eta}{w}},
\end{equation*}
with equality if and only if the numbers \(\dualproduct{\eta}{v}\) and \(\dualproduct{\eta}{w}\) have the same sign.
We have 
\begin{equation*}
\begin{split}
 \norm{\targetspace}{v + w}
 &= \int_{\targetspace'} \bigabs{ \dualproduct{\eta}{v}  + \dualproduct{\eta}{w}} \dif\sigma(\eta)\\
 &\ge \int_{\targetspace'}\bigabs{\dualproduct{\eta}{v}} \dif\sigma(\eta)
+  \int_{\targetspace'}\bigabs{\dualproduct{\eta}{w}} \dif\sigma(\eta) = \norm{\targetspace}{v}  + \norm{\targetspace}{w},
\end{split}
\end{equation*}
with equality if and only if \(\dualproduct{\eta}{v}\) and \(\dualproduct{\eta}{w}\) have the same sign for every \(\eta \in \operatorname{supp} \sigma\), since the map \(z \in V' \mapsto \dualproduct{\eta}{z}\) is continuous.

In particular, if \(\sigma\) is a non-degenerate zonal representation, then the norm \(\norm{\targetspace}{\cdot}\) is strictly convex.
In the two-dimensional case \(\dim \targetspace = 2\), the separation condition of \cref{proposition_strict_convexity_zonal} on the support conversely
implies the density condition on the support. 
Indeed, for every \(\eta \in \targetspace'\) and every \(\delta > 0\), there exists \(v, w \in \targetspace\) such that if \(\eta'\in \targetspace'\) and \(\dualproduct{\eta'}{v} \le 0 \le \dualproduct{\eta'}{w}\), then \(\norm{\targetspace'}{\eta - \eta'} \le \delta\).
\end{proof}

Density of the support for a zonal representation of a strictly
convex norm in higher dimensions \(\dim\targetspace\geq 3\) fails
in general. As an example, one can consider \(\targetspace=\reals^3\)
and the measure
  \( \sigma = \sigma_{x,y} + \sigma_{x,z} + \sigma_{z,y} \),
where \(\sigma_{x,y}\) is the uniform measure on the circle orthogonal to the \(y\)-axis; and similarly for \(\sigma_{x,z}\) and \(\sigma_{z,y}\).
Then the support \(\support(\sigma)\) is far from being dense in \(\sphere\),
but the induced norm
  \begin{equation*} \int_{\sphere} \bigabs{\dualproduct{\eta}{v}} \dif\sigma(\eta) = \sqrt{{v_1}^2+{v_2}^2}
			+ \sqrt{{v_1}^2+{v_3}^2}
			+ \sqrt{{v_2}^2+{v_3}^2} .
  \end{equation*}
is strictly convex.

\bigskip

We know give another proof of~\cref{theorem_Euclidean} by mean of
zonal integral representation. We restrict ourselves to the case of Euclidean
norms in the case \(\dim\targetspace\geq 3\), and strictly convex norms in
the case \(\dim\targetspace=2\).

\begin{proof}[Second proof of~\cref{theorem_Euclidean}, by zonal representation]
We shall only need to prove that mass convergence implies convergence of the range,
as the converse statement holds in all generality.

For every \(\eta \in \targetspace'\), the sequence \((\dualproduct{\eta}{\mu_n})_{n \in \integers}\) converges widely to \((\dualproduct{\eta}{\mu})\) and by \cref{proposition_variation_wlsc}, we have 
\begin{equation}
\label{eq_WaNguuce8o}
  \liminf_{n \to \infty} 
    \variation{\reals}{\dualproduct{\eta}{\mu_n}} (X)
  \ge  \variation{\reals}{\dualproduct{\eta}{\mu}} (X).
\end{equation}
Since \(\norm{\targetspace}{\cdot}\) is a Euclidean norm, there exists a constant \(\kappa > 0\) such that for every \(v \in \targetspace\), 
\begin{equation}
\label{eq_zonal_euclidean}
    \norm{\targetspace}{v}
  =
    \kappa \int_{\sphere}
      \abs{\dualproduct{\eta}{v}}
      \dif \eta
\end{equation}
where \(\sphere\) is the unit sphere in \(\targetspace \simeq \targetspace'\).
By \cref{lemma_mass_representation} and by our assumption, we have 
\begin{equation}
\label{eq_quie4ohg6J}
  \limsup_{n \to \infty} \int_{\sphere}
      \variation{\reals}{\dualproduct{\eta}{\mu_n}} (X)
      \dif \eta
      \le  \int_{\sphere}
      \variation{\reals}{\dualproduct{\eta}{\mu}} (X)
      \dif \eta.
\end{equation}
For every \(\eta_1, \eta_2 \in \targetspace'\) and \(n \in \integers\), we have
\begin{equation}
\label{eq_Ohr8giaGhe}
  \bigabs{
    \variation{\reals}{\dualproduct{\eta_2}{\mu_n}}(X)
    -
    \variation{\reals}{\dualproduct{\eta_1}{\mu_n}}(X)
  }
  \le 
   \norm{\targetspace'}{\eta_2 - \eta_1}
   \variation{\targetspace}{\mu_n} (X).
\end{equation}

By \eqref{eq_WaNguuce8o}, \eqref{eq_quie4ohg6J} and \eqref{eq_Ohr8giaGhe}, we deduce that for every \(\eta \in \mathbb{S}_{\targetspace'}\), we have 
\begin{equation}
  \lim_{n \to \infty} 
    \variation{\reals}{\dualproduct{\eta}{\mu_n}} (X)
  = \variation{\reals}{\dualproduct{\eta}{\mu}} (X).
\end{equation}
By homogeneity, this still holds for every \(\eta \in \targetspace'\) and the conclusion follows then from \cref{lemma_pointwise_to_hausdorff}.
\end{proof}

\section{Perimeter representation of masses}

\subsection{Perimeter formula}

So far, we have used zonal representation of norms in spaces of dimension~2
to approach the problem of mass convergence from a geometric point of view.
In particular, we were able to characterize strict convexity of the norm
through density of the support for the measure representing the norm.

We now give further geometric insight of the problem, by interpreting the
total mass of a vector valued measure as the perimeter of its range
(up to some multiplicative constant).

We first define the perimeter of a polygon in a two-dimensional space.

\begin{definition}
If \(\dim \targetspace = 2\) and \(P \subset V\)
if \(P\) is a polygon of vertices \(v_1, \dotsc, v_\ell\), we define its perimeter
as 
\begin{equation*}
  \operatorname{Per}_\targetspace(P)
  \defeq 
  \sum_{i = 1}^{\ell} \norm{\targetspace}{v_{\ell} - v_{\ell - 1}},
\end{equation*}
with the convention that \(v_0 = v_\ell\).
\end{definition}

This notion of perimeter is analogous to the notion of perimeter in crystalline variational problems \citelist{\cite{Taylor_1974}\cite{Dinghas_1944}\cite{Wulff_1901}\cite{Figalli_Maggi_Pratelli_2010}}.

The perimeter of convex polygons is monotone with respect to set inclusion.

\begin{lemma}
\label{proposition_perimeter_monotone_polygon}
If \(\dim \targetspace = 2\) and \(P_1 \subseteq P_2 \subset V\) are convex polygons, then 
\begin{equation*}
  \operatorname{Per}_\targetspace(P_1) \le \operatorname{Per}_\targetspace(P_2). 
\end{equation*}
\end{lemma}
\begin{proof}
Let \(H \subset \targetspace\) be a half-plane such that \(P_1 \subset H\) and \(\partial H\) contains one side of \(\partial P_1\). Then by the triangle inequality for the seminorm \(\operatorname{Per}_\targetspace (H \cap P_2) \le \operatorname{Per}_\targetspace (H)\).
By repeating this process on each edge of the polygon \(P_1\), we reach the conclusion.
\end{proof}

\Cref{proposition_perimeter_monotone_polygon} suggests to extend the definition of perimeter to convex sets as the least upper bound of the perimeter of polygonal subsets.

\begin{definition}
\label{definition_perimeter_convex}
If \(C \subset V\) is convex, we set
\begin{equation*} 
  \operatorname{Per}_\targetspace(C) \defeq \sup\,\bigl\{\operatorname{Per}_\targetspace (P) :  P \subseteq C\text{ is a convex polygon} \Bigr\} .
\end{equation*} 
\end{definition}

In view of \cref{proposition_perimeter_monotone_polygon}, \cref{proposition_perimeter_monotone_polygon} and \cref{definition_perimeter_convex} coincide for polygons.

The monotonicity property of \cref{proposition_perimeter_monotone_polygon} can be transferred to arbitrary subsets and the equality cases can be characterized for strictly convex norms.

\begin{proposition}\label{PerimeterMonotony}
If  \(C_1\subseteq C_2\subseteq \targetspace\) are compact convex sets, then 
\begin{equation*}
\operatorname{Per}_\targetspace(C_1)\leq\operatorname{Per}_\targetspace(C_2), 
\end{equation*}
If \(\norm{\targetspace}{\cdot}\) is a strictly convex norm and if equality holds, then \(C_1=C_2\).
\end{proposition}
\begin{proof}
The inequality follows from \cref{proposition_perimeter_monotone_polygon} and \cref{definition_perimeter_convex}.

For the equality case, we assume by contradiction that \(C_1 \subsetneq C_2\)
There exists \(v \in C_2 \setminus C_1\). 
Let \(H \subset V\) be a half-plane such that \(v \not \in H\) and \(C_1 \subset H\). In particular, for every \(w_1, w_2 \in C_1\), we have since \(\norm{\targetspace}{\cdot}\) is a strictly convex norm,
\begin{equation*}
 \norm{\targetspace}{w_2 - w_1}
 < \norm{\targetspace}{w_2 - v} + \norm{\targetspace}{v - w_1}.
\end{equation*}
We define 
\begin{equation*}
 \eta 
 \defeq
 \inf \{ \norm{\targetspace}{w_2 - v} + \norm{\targetspace}{v - w_1} - \norm{\targetspace}{w_2 - w_1} \st w_1, w_2 \in C_1\}.
\end{equation*}
Since the set \(C_1\) is compact and convex, we have \(\eta > 0\). If \(P_1 \subseteq C_1\) is a polygon, then the polygon \(P_2\) defined as the convex hull of \(P_1\) and \(v\) satisfies \(P_2 \subseteq C_2\) and 
\begin{equation*}
 \operatorname{Per}_{\targetspace} (P_2)
 \ge \operatorname{Per}_{\targetspace} (P_1) + \eta.
\end{equation*}
By taking the supremum, it follows from \cref{definition_perimeter_convex} that 
\begin{equation*}
 \operatorname{Per}_{\targetspace} (C_2)
 \ge \operatorname{Per}_{\targetspace} (C_1) + \eta
 > \operatorname{Per}_{\targetspace} (C_1).
 \qedhere
\end{equation*}
\end{proof}

We establish \cref{theorem_perimeter_computation} which states that 
\begin{equation*} 
  \mass{\targetspace}{\mu} = \tfrac{1}{2}\operatorname{Per} \bigl(\range{\mu}\bigr) .
\end{equation*}
This formula together with \cref{lemma_mass_representation}
yield a Cauchy--Crofton-like integral formula for the perimeter: 
for any compact convex set \(C\subseteq\targetspace\):
\begin{equation}
  \operatorname{Per}_\targetspace(C) = 2 \int_{\sphere} \diam \dualproduct{\eta}{C}\,\dif\sigma(\eta) . 
\end{equation}

The proof of \cref{theorem_perimeter_computation} is based on the following property of polygons.

\begin{lemma}
\label{lemma_perimeter_polygon}
Let \(V\) be a two-dimensional vector space and \(\norm{\targetspace}{\cdot}\) be a seminorm on \(V\).
For every \(m \in \integers\) and \(v_1, \dotsc, v_m \in \targetspace\), one has 
\begin{equation*}
    \sum_{n = 1}^m \norm{\targetspace}{v_n} 
 = 	
    \tfrac{1}{2} \operatorname{Per}_{\targetspace} \bigl( [0, v_1] + \dotsb + [0, v_m]\bigr).
\end{equation*}
\end{lemma}

\Cref{lemma_perimeter_polygon} implies in particular immediately that \cref{theorem_perimeter_computation} holds when \(\mu\) has finite support.

\begin{proof}[Proof of \cref{lemma_perimeter_polygon}]
We proceed by induction. The identity is trivial if \(m = 0\). The induction step follows from the observation that 
\begin{equation*}
 \tfrac{1}{2} \operatorname{Per}_{\targetspace} \bigl( [0, v_1] + \dotsb + [0, v_{m + 1}]\bigr)
 =
  \tfrac{1}{2} \operatorname{Per}_{\targetspace} \bigl( [0, v_1] + \dotsb + [0, v_{m}]\bigr)
  + \norm{\targetspace}{v_{m + 1}}.
  \qedhere
\end{equation*}
\end{proof}

\begin{proof}[Proof of \cref{theorem_perimeter_computation}]
If \(E_1, \dotsc, E_m \in \Sigma\) are disjoint subsets of \(A\), then we have 
\begin{equation*}
  \sum_{n = 1}^m \norm{\targetspace}{\mu (E_n)}
  = \tfrac{1}{2} \operatorname{Per}_{\targetspace} \bigl( [0, \mu (E_1)] + \dotsb + [0, \mu (E_n)]\bigr).
\end{equation*}
Since \([0, \mu (E_1)] + \dotsb + [0, \mu (E_n)] \subset \range[A]{\mu}\), it follows from the definition of total variation (\cref{definition_mass}) that 
\begin{equation}
\label{eq_coosheL3ee}
 \variation{\targetspace}{\mu}(A)
 \le \tfrac{1}{2} \operatorname{Per}_{\targetspace} (\range[A]{\mu}).
\end{equation}

Conversely, if \(P \subset \range[A]{\mu}\) is a polygon, there exist sets \(F_1, \dotsc, F_k \in \Sigma\) such that \(F_1, \dotsc, F_k \subseteq A\),
\(\mu (F_1), \dotsc, \mu (F_k)\) are the vertices of \(P\). Hence there exists disjoint sets \(E_1, \dotsc, E_m \in \Sigma\) such that \(P \subseteq 
[0, \mu (E_1)] + \dotsb + [0, \mu (E_m)]\).
It follows thus by \cref{definition_perimeter_convex} that 
\begin{equation}
 \tfrac{1}{2} \operatorname{Per}_{\targetspace} (P)
  \le 
    \tfrac{1}{2} \operatorname{Per}_{\targetspace} \bigl( [0, \mu (E_1)] + \dotsb + [0, \mu (E_m)]\bigr)
 =
  \sum_{n = 1}^m \norm{\targetspace}{\mu (E_n)} (A).
\end{equation}
Hence by definition of total variation (\cref{definition_mass}),
\begin{equation*}
\tfrac{1}{2} \operatorname{Per}_{\targetspace} (P)
\le \variation{\targetspace}{\mu}(A)
\end{equation*}
and thus by definition of perimeter (\cref{definition_perimeter_convex})
\begin{equation}
\label{eq_ooph1ZahZi}
    \tfrac{1}{2} \operatorname{Per}_{\targetspace} (\range[A]{\mu})
  \le 
    \variation{\targetspace}{\mu}(A).
\end{equation}

The conclusion comes from the combination of \eqref{eq_coosheL3ee} and \eqref{eq_ooph1ZahZi}.
\end{proof}

\subsection{Application to monotonicity and convergence}
\Cref{theorem_perimeter_computation} gives a proof of the monotonicity of the range with respect to the mass.

\begin{proof}[Second proof of \cref{mass_monotony} when \(\dim V = 2\)]
  We apply \cref{theorem_perimeter_computation} and \cref{PerimeterMonotony}.
\end{proof}

As a final result, we give a proof of~\cref{theorem_Euclidean}
in the two-dimensional case, based on the perimeter formula,
\cref{theorem_perimeter_computation}.

\begin{proof}[Third proof of~\cref{theorem_Euclidean} by mean of perimeter formula]
If the set \(C \subset V\) is compact and convex, we define the function \(\gamma_C : \targetspace \to [0, + \infty)\) by 
\begin{equation}
\label{eq_eiGuul3Aip}
  \gamma_C (v) 
  \defeq
  \operatorname{Per}_{\targetspace} \bigl(\chull{(C + \{v\})}\bigr).
\end{equation}
We observe that, by the triangle inequality and the definitions, for every \(v_1, v_2 \in \targetspace\),\begin{equation}
\label{eq_euT7phu9gu}
\abs{\gamma_C (v_2) - \gamma_C (v_1)}
\le 2 \norm{\targetspace}{v_2 - v_1}
\end{equation}
and that if \(0 \in C\),
\begin{equation}
\label{eq_zohB1veec7}
 \norm{\targetspace}{v} \le \gamma_C (v).
\end{equation}

Let \(\varepsilon > 0\). 
By the Weierstrass extreme value theorem and, since \(\norm{\targetspace}{\cdot}\) is a strictly convex norm, by \cref{PerimeterMonotony} we have 
\begin{equation*}
 \delta \defeq \inf \{\gamma_{\range{\mu}} (v) \st \operatorname{dist} (v, \range{\mu}) \ge \varepsilon \}
 - \operatorname{Per}_{\targetspace} (\range{\mu})
 > 0.
\end{equation*}
By \eqref{eq_eiGuul3Aip} and \cref{definition_perimeter_convex} for every \(v \in \targetspace\), there exists a polygon \(P_v \subset \range{\mu}\) such that 
\begin{equation*}
 \gamma_{P_v} (v) \ge \gamma_{\range{\mu}} (v) - \frac{\delta}{4}.
\end{equation*}
Thus if \(\norm{\targetspace}{v' - v} \le \frac{\delta}{16}\), we have by \eqref{eq_euT7phu9gu}
\begin{equation*}
 \gamma_{P_v} (v') \ge \gamma_{\range{\mu}} (v') - \frac{\delta}{2}
\end{equation*}
and thus if \(\gamma_{P_v} (v') \le \operatorname{Per}_\targetspace (\range \mu) + \frac{\delta}{2}\), then \(v' \in \range{\mu} + \closure{B (0, \varepsilon)}\).
In view of \eqref{eq_zohB1veec7}, by a compactness argument, there exists a polygon \(P \subseteq \range{\mu}\) such that if \(\gamma_{P} (v) \le \operatorname{Per}_\targetspace (\range \mu) + \frac{\delta}{2}\), then \(v \in \range{\mu} + \closure{B (0, \varepsilon)}\).
  By \cref{proposition_wlsc_range} and \cref{definition_Kuratowski_lower_limit}, there exists a sequence of polygons \((P_n)_{n \in \integers}\) whose vertices converge to those of \(P\) and such that  for every \(n \in \integers\), \(P_n \subset \range{\mu_n}\). 
For every \(v \in \targetspace\), we have \(\lim_{n \to \infty} \gamma_{P_n} (v) = \gamma_{P} (v)\), and thus in view of \eqref{eq_euT7phu9gu}, for \(n \in \integers\) large enough, if \(v \in \targetspace\) satisfies \(\gamma_{P_n} (v) \le \operatorname{Per}_\targetspace (\range \mu) + \frac{\delta}{4}\), then \(v \in \range{\mu} + \closure{B (0, \varepsilon)}\).

Assume now that \(v \in \range{\mu_n}\). One has then \(\chull{P_n + \{v\}} \subset \range{\mu_n}\), and thus if \(n \in \integers\) is large enough, we have  
\begin{equation*}
 \gamma_{P_n} (v) \le \operatorname{Per}_{\targetspace} (\mu_n)
 \le \operatorname{Per}_{\targetspace} (\mu) + \frac{\delta}{4}.
\end{equation*}
By construction of \(P_n\), this implies that \(v \in \range{\mu} + \closure{B (0, \varepsilon)}\) and thus 
\begin{equation}
\label{eq_hee1aib2Be}
 \range{\mu_n} \subseteq \range{\mu} +\closure{B (0, \varepsilon)}.
\end{equation}

The conclusion then follows from \eqref{eq_hee1aib2Be} and \cref{proposition_wlsc_range}.
\end{proof}

\section*{Acknowledgement}

The authors thank Jan Kristensen for bringing to their attention the references \citelist{\cite{Delladio}\cite{Reshetnyak}}.

\end{document}